\documentclass[letterpaper , 12pt ,reqno]{article}

\usepackage{fullpage}		

\usepackage{amsmath, amsthm}
\usepackage{eucal}
\usepackage{mathtools}
\usepackage{verbatim}
\usepackage{amssymb}
\usepackage{latexsym}
\usepackage{graphicx}
\usepackage{amsfonts}
\usepackage{cases}

\usepackage{bbding}

\allowdisplaybreaks     

\usepackage[dvipsnames]{xcolor}
\usepackage{colortbl}	

\usepackage{url}

\usepackage{float}		

\usepackage{tikz}			
\usetikzlibrary{cd,positioning, shapes, arrows}		

\usepackage{tkz-graph}

\usepackage{mdframed}		

\usepackage{empheq}
\usepackage[most]{tcolorbox}

\newtcbox{\mymath}[1][]{%
    nobeforeafter, math upper, tcbox raise base,
    enhanced, colframe=blue!30!black,
    colback=blue!30, boxrule=1pt,
    #1}

\usepackage{subfigure}	

\usepackage{soul}		

\usepackage[autolanguage]{numprint}  

\usepackage{bm}		

\usepackage{centernot}  

\usepackage[toc]{appendix}	

\usepackage{chngcntr}		

\usepackage{mathrsfs}		

\usepackage{enumerate}		

\usepackage{wasysym}        

\usepackage{cancel}         

\definecolor{shadecolor}{gray}{0.90}				
\def\boitegrise#1#2{\begin{centerline}{\fcolorbox{black}{shadecolor}{~
    \begin{minipage}[t]{#2}{\vphantom{~}#1\vphantom{$A_{\displaystyle{A_A}}$}}
            \end{minipage}~}}\end{centerline}\medskip}

\usepackage[shortlabels]{enumitem}	

\setenumerate[1]{label=\thesubsection.\arabic*.}
\setenumerate[2]{label*=\arabic*.}

\usepackage{multicol}

\usepackage{tablefootnote}

\usepackage[colorlinks=true, citecolor=blue, linktocpage]{hyperref}

\DeclareMathOperator{\ord}{ord}

\newcommand{\pmodd}[1]{\!\!\!\pmod{#1}}	
\newcommand{\FibSeq}{\left(F_n\right)_{n \geq 0}}
\newcommand{\LucSeq}{\left(L_n\right)_{n \geq 0}}

\newcommand{\PellSeq}{\left(P_n\right)_{n \geq 0}}
\newcommand{\QellSeq}{\left(Q_n\right)_{n \geq 0}}
\newcommand{\BalSeq}{\left(B_n\right)_{n \geq 0}}

\newcommand{\LucBalSeq}{\left(C_n\right)_{n \geq 0}}

\newcommand{\LucasFormalUSeq}{\left(U_n(p,q)\right)_{n \geq 0}}
\newcommand{\LucasFormalVSeq}{\left(V_n(p,q)\right)_{n \geq 0}}

\newcommand{\LucasFormalUTerm}{U_n(p,q)}
\newcommand{\LucasFormalVTerm}{V_n(p,q)}
\newcommand{\LucasUSeq}{\left(U_n\right)_{n \geq 0}}
\newcommand{\LucasVSeq}{\left(V_n\right)_{n \geq 0}}

\newcommand{\GenericSeq}{\left(S_n\right)_{n \geq 0}}

\newcommand{\seqnum}[1]{\href{https://oeis.org/#1}{\rm \underline{#1}}}

\DeclareMathOperator{\newF}{\mathcal{F}}

\newcommand{\Ften}{\left(\newF_{10,n}\right)_{n=0}^\infty}
\newcommand{\F}[1]{\newF_{10, #1}}

\newcommand{\Ftensub}{\left(\newF_{10,k+rj}\right)_{j=0}^\infty}

\DeclareFontFamily{U}{BOONDOX-calo}{\skewchar\font=45 }
\DeclareFontShape{U}{BOONDOX-calo}{m}{n}{
  <-> s*[1.05] BOONDOX-r-calo}{}
\DeclareFontShape{U}{BOONDOX-calo}{b}{n}{
  <-> s*[1.05] BOONDOX-b-calo}{}
\DeclareMathAlphabet{\mathcalb}{U}{BOONDOX-calo}{m}{n}
\SetMathAlphabet{\mathcalb}{bold}{U}{BOONDOX-calo}{b}{n}
\DeclareMathAlphabet{\mathbcalb}{U}{BOONDOX-calo}{b}{n}

\newcommand{\red}[1]{\textcolor{red}{#1}}
\newcommand{\blue}[1]{\textcolor{blue}{#1}}

\newcommand{\redbf}[1]{\textcolor{red}{\textbf{#1}}}

\newcounter{ga} 

\newtheorem{theorem}{Theorem}[section] 
\newtheorem{lemma}[theorem]{Lemma}
\newtheorem{corollary}[theorem]{Corollary}
\newtheorem{conjecture}[theorem]{Conjecture}
\newtheorem*{conjecture*}{Conjecture}

\newtheorem{proposition}[theorem]{Proposition}

\theoremstyle{definition}
\newtheorem{example}[theorem]{Example}

\theoremstyle{definition}
\newtheorem{definition}[theorem]{Definition}

\theoremstyle{plain}

\theoremstyle{definition}
\newtheorem{question}[theorem]{Question}

\theoremstyle{definition}
\newtheorem{remark}[theorem]{Remark}

\theoremstyle{definition}
\newtheorem{convention}[theorem]{Convention}

\numberwithin{equation}{section}

\begin{document}

\title{\redbf{Period patterns, entry points, and orders in the Lucas sequences: theory and applications}}

\author{
{\large Morgan Fiebig} \\  
\href{mailto:fiebigm8842@uwec.edu}{\small\color{Melon}\nolinkurl{fiebigm8842@uwec.edu}}
\and
{\large aBa Mbirika\footnote{corresponding author}}\\
\href{mailto:mbirika@uwec.edu}{\small\color{Melon}\nolinkurl{mbirika@uwec.edu}}
\and
{\large J\"urgen Spilker} \\  
\href{mailto:juergen.spilker@t-online.de}{\small\color{Melon}\nolinkurl{juergen.spilker@t-online.de}}
}

\date{}

\makeatletter
\newcommand{\subjclass}[2][2020]{%
  \let\@oldtitle\@title%
  \gdef\@title{\@oldtitle\footnotetext{#1 \emph{Mathematics Subject Classification.} #2}}%
}
\newcommand{\keywords}[1]{%
  \let\@@oldtitle\@title%
  \gdef\@title{\@@oldtitle\footnotetext{\emph{Key words and phrases.} #1.}}%
}
\makeatother


\maketitle

\begin{abstract}
The goal of this paper is twofold: (1) extend theory on certain statistics in the Fibonacci and Lucas sequences modulo $m$ to the Lucas sequences $U := \left(U_n(p,q)\right)_{n \geq 0}$ and $V := \left(V_n(p,q)\right)_{n \geq 0}$, and (2) apply some of this theory to a novel graphical approach of $U$ and $V$ modulo $m$. Upon placing the cycle of repeating sequence terms in a circle, several fascinating patterns which would otherwise be overlooked emerge.  We generalize a wealth of known Fibonacci and Lucas statistical identities to the $U$ and $V$ settings using primary sources such as Lucas in 1878, Carmichael in 1913, Wall in 1960, and Vinson in 1963, amongst others. We use many of these generalized identities to form the theoretical basis for our graphical results. Based on the order of $m$, defined as $\omega(m) := \frac{\pi(m)}{e(m)}$, where $\pi(m)$ is the period of $m$ and $e(m)$ is the entry point of $m$, we describe behaviors shared by $U$ and $V$ with parameters $q = \pm 1$. In particular, we exhibit some tantalizing examples in the following three sequence pairs: Fibonacci and Lucas, Pell and associated Pell, and balancing and Lucas-balancing.
\end{abstract}

\tableofcontents



\section{Introduction and motivation}\label{sec:motivation}


\subsection{Introduction: some history and the goal of this paper}
In this paper we conduct a theory and application based study of the infinite families of Lucas sequences $\LucasFormalUSeq$ and $\LucasFormalVSeq$ modulo $m$. The Fibonacci sequence modulo $m$ has been well-studied since the seminal work of Wall in 1960~\cite{Wall1960}. As of this writing, the database \texttt{MathSciNet} points to 126 citations to this paper alone. The first two issues of the \textit{Fibonacci Quarterly} in 1963 introduced the foundational work on this subject by Vinson~\cite{Vinson1963} and Robinson~\cite{Robinson1963}. However, this subject has roots that go as far back as the original ground-breaking papers by Lucas in 1878~\cite{Lucas1878} and Carmichael in 1913/14 and 1920~\cite{Carmichael1913, Carmichael1920}. More recently, innovative extensions of research into the Lucas sequences modulo $m$ has been done by Desmond in 1978~\cite{Desmond1978} and Renault in 1996 and 2013~\cite{Renault1996, Renault2013}. In this paper we utilize work from all of these aforementioned primary sources to prove our main results.

The goal of this paper is twofold: (1) extend theory on certain statistics in the Fibonacci $\FibSeq$ and Lucas $\LucSeq$ sequences modulo $m$ to the Lucas sequences $\left(U_n(p,q)\right)_{n \geq 0}$ and $\left(V_n(p,q)\right)_{n \geq 0}$ modulo $m$, and (2) apply some of this theory to a novel graphical approach of the Lucas sequences modulo $m$. The statistics we explore are the period $\pi(m)$, entry point $e(m)$, and order $\omega(m) := \frac{\pi(m)}{e(m)}$. Based on $\omega(m)$, we describe behaviors shared by infinite families of nondegenerate Lucas sequences with parameters $q = \pm 1$.


\subsection{Motivation: the Fibonacci sequence modulo 10}\label{subsec:motivation}

This paper is motivated by previous work of second-author Mbirika and collaborators Guyer and Scott~\cite{Guyer_Mbirika_Scott2024}. They introduced a graphical approach to the Fibonacci sequence modulo $10$ (denoted $\Ften$ where $\newF_{10,n}$ is the least nonnegative residue of $F_n$ modulo 10). It is well known that $\FibSeq$ modulo $m$ is purely periodic for every $m \geq 1$. For $m=10$, this period has length 60, and the 60 sequence values $F_0, \ldots, F_{59}$ modulo 10 can be pictorially represented by equally spacing the residue classes $\F{0}, \ldots, \F{59}$ clockwise around the circumference of the circle, starting with $\F{0}$ at the top, as in Figure~\ref{fig:FibTen_circle}.
\begin{figure}[h]
\centering
\resizebox{3in}{!}
{
\begin{tikzpicture}
\draw[blue, thick] (0,0) circle (2in);
\foreach \x in {6,12,...,360} {
\draw (\x: 2in) -- (\x: 2.1in);}
\foreach \x/\t in {0/0,6/7,12/3,18/4,24/9,30/5,36/4,42/1,48/3,54/8,60/5,66/3,72/2,78/1,84/1,90/0,96/1,102/9,108/2,114/7,120/5,126/2,132/3,138/9,144/4,150/5,156/9,162/6,168/3,174/3,180/0,186/3,192/7,198/6,204/1,210/5,216/6,222/9,228/7,234/2,240/5,246/7,252/8,258/9,264/9,270/0,276/9,282/1,288/8,294/3,300/5,306/8,312/7,318/1,324/6,330/5,336/1,342/4,348/7,354/7} {
\draw (\x: 2.2in) node{\t};}
\foreach \x in {1,...,6} {
\draw[red] (\x*30: 2in) -- (\x*30+180: 2in);}
\foreach \x/\t in {0/90,1/60,2/30,3/0,4/330,5/300,6/270,7/240,8/210,9/180,10/150,11/120} {
\draw (\x*30: 1in) node[fill=white]{$\text{\t}^\circ$};}
\end{tikzpicture}
}
\caption{The Pisano period $\left(\F{n}\right)_{n=0}^{59}$ with values equally spaced}
\label{fig:FibTen_circle}
\end{figure}

In particular, they investigated subsequences $\Ftensub$ composed of equally spaced terms of the sequence $\Ften$; that is, every $r^{\mathrm{th}}$ term of $\Ften$ starting from the term $\F{k}$ for some $0 \leq k \leq 59$. For example if $r=30$, the subsequence $\left(\newF_{10,k+30j}\right)_{j=0}^\infty$ is comprised of a single pair of alternating antipodal points on the circle. Moreover, for this specific $r$ value, they proved that for all $n \in \mathbb{Z}$, we have
\begin{align}
    F_{n} +F_{n+\frac{\pi_{F}(10)}{2}} \equiv 0 \pmodd{10},\label{eq:Dan_aBa_Miko_antipodal_result}
\end{align}
where $\pi_F(10)$ denotes the period of $\FibSeq$ modulo 10~\cite[Theorem~5.3]{Guyer_Mbirika_Scott2024}. In other words, antipodal points are additive inverses of each other modulo $10$. A natural question to ask is whether this holds for arbitrary Lucas sequences $\LucasFormalUSeq$.


\subsection{The graphical approach extended to the Lucas sequences}

In this current paper, we find patterns that not only hold for infinitely many $m$ values, but also hold for infinite families of Lucas sequences $\LucasFormalUSeq$ and $\LucasFormalVSeq$. For example, we generalize the Guyer-Mbirika-Scott result~\cite[Theorem~5.3]{Guyer_Mbirika_Scott2024} given in Equation~\eqref{eq:Dan_aBa_Miko_antipodal_result} in the previous subsection. We prove that this holds for all $m>2$ in the sequences $\LucasFormalUSeq$ and $\LucasFormalVSeq$ with $q = -1$ that satisfy a certain order condition (see Theorems~\ref{thm:antipodal_result_for_U_n_sequence_with_order_4} and \ref{thm:antipodal_result_for_V_n_sequence_with_order_4}). Another intriguing pattern we reveal is palindromes in the Lucas-balancing sequence $\LucBalSeq$, equivalently the sequence $\left( \frac{1}{2} V_n(6,1) \right)_{n \geq 0}$. This behavior is unique in that the other ``Lucas''-variants of the popular sequences Fibonacci $\FibSeq$ and Pell $\PellSeq$, namely Lucas $\LucSeq$ and associated Pell $\QellSeq$, respectively, do not exhibit this palindromic behavior. Strikingly, it is not just $\LucBalSeq$ that has this palindromic behavior, it is the infinitely many sequences $\left(V_n(p,1)\right)_{n \geq 0}$ that exhibit palindromes for all moduli $m > 2$ (see Theorem~\ref{thm:V_n_with_q_equal_1_palindrome_result}).

The breakdown of this paper is as follows. In Section~\ref{sec:definitions}, we provide the definitions of all the sequences we consider and give relevant identities for $\LucasFormalUSeq$ and $\LucasFormalVSeq$ which we use throughout the paper. Section~\ref{sec:theory} is the theory portion where we discuss a few well-known results and deliver new results on the statistics of period $\pi(m)$, entry point $e(m)$, and order $\omega(m)$.
Section~\ref{sec:applications} is the applications portion where, using a graphical approach, we present a host of intriguing patterns in the fundamental periods of the Lucas sequences. Finally in Section~\ref{sec:open questions}, we provide open questions for further research for the motivated reader.

\begin{remark}
Note that although much of Section~\ref{sec:theory} is developed to prove the main results in Section~\ref{sec:applications}, many of the theorems we present in Section~\ref{sec:theory} are interesting in their own right and are not explicitly in the literature, as far as we know.
\end{remark}


\section{Definitions and preliminaries}\label{sec:definitions}

\subsection{Lucas sequences: \texorpdfstring{$\LucasFormalUSeq$}{(Un)} and \texorpdfstring{$\LucasFormalVSeq$}{(Vn)}}\label{subsec:Lucas_sequences}

\begin{definition}[Lucas sequences]
Let $p$ and $q$ be nonzero integers with $\gcd(p,q)=1$, and let $\alpha$ and $\beta$ be roots of the \textit{characteristic polynomial} $x^2 - px + q$ in the quadratic field $\mathbb{Q}(\sqrt{\Delta})$ where $\Delta = p^2 - 4q$ is the \textit{discriminant} of the polynomial. The following identities hold:
\begin{multicols}{3}
\begin{enumerate}[(i)]
    \item[] $\alpha = \frac{p + \sqrt{\Delta}}{2}$ 
    \item[] $\beta = \frac{p - \sqrt{\Delta}}{2}$
    \item[] $\alpha+\beta = p$
    \item[] $\alpha\beta = q$
    \item[] $\alpha-\beta = \sqrt{\Delta}$
    \item[]
\end{enumerate}
\end{multicols}
\noindent The \textit{Lucas sequences with parameters $p$ and $q$}, denoted $\LucasFormalUSeq$ and $\LucasFormalVSeq$, respectively, are given by
$$ U_n(p,q) = \frac{\alpha^n - \beta^n}{\alpha - \beta} \;\;\;\text{and}\;\;\; V_n(p,q) = \alpha^n + \beta^n.$$
Setting $U_n := U_n(p,q)$ and $V_n := V_n(p,q)$, the sequences $\LucasUSeq$ and $\LucasVSeq$ satisfy the recurrence relations $U_n = p U_{n-1} - q U_{n-2}$ and $V_n = p V_{n-1} - q V_{n-2}$, respectively, with initial conditions $U_0 = 0$, $U_1 = 1$, $V_0 = 2$, and $V_1=p$.
\end{definition}

\boitegrise{
\begin{convention}\label{conv:nondegenerate_Lucas_sequences}
\vspace{-.15in}
In this paper, we assume that  $\LucasFormalUSeq$ and $\LucasFormalVSeq$ are \textit{nondegenerate}. That is, $q \neq 0$ and the ratio $\frac{\alpha}{\beta}$ is not a root of unity. In particular, this implies that $\alpha$ and $\beta$ are distinct
and hence $\Delta \neq 0$. For $\gcd(p,q)=1$, a Lucas sequence is degenerate if $(p,q) \in \{(\pm 2,1)$, $(\pm 1,1)$, $(0,\pm 1)$, and $(\pm 1,0)\}$~\cite[pp.~5--6]{Ribenboim2000}.
\vspace{-.3in}
\end{convention}}{0.95\textwidth}


\subsection{Sequences: \texorpdfstring{$\FibSeq, \LucSeq, \PellSeq, \QellSeq, \BalSeq, \text{and} \LucBalSeq$}{Fn, Ln, Pn, Qn, Bn, Cn}}\label{subsec:particular_sequences}
We recall the recursive definitions of the six particular second-order linear recurrence sequences considered in this paper.

\begin{definition}\label{def:Fib_Luc_numbers}
The \textit{Fibonacci sequence} $\FibSeq$ and \textit{Lucas sequence} $\LucSeq$ are defined by the recurrence relations $F_n = F_{n-1} + F_{n-2}$ and $L_n = L_{n-1} + L_{n-2}$, respectively,
with initial conditions $F_0 = 0$, $F_1 = 1$, $L_0 = 2$, and $L_1 = 1$. In the OEIS, these are sequences \seqnum{A000045} and \seqnum{A000032}~\cite{Sloane-OEIS}.
\end{definition}

\begin{definition}\label{def:Pell_Qell_numbers}
The \textit{Pell sequence} $\PellSeq$ and \textit{associated Pell sequence} $\QellSeq$ are defined by the recurrence relations $P_n = 2 P_{n-1} + P_{n-2}$ and $Q_n = 2 Q_{n-1} + Q_{n-2}$, respectively,
with initial conditions $P_0 = 0$, $P_1 = 1$, $Q_0 = 1$, and $Q_1 = 1$. In the OEIS, these are sequences \seqnum{A000129} and \seqnum{A001333}~\cite{Sloane-OEIS}.
\end{definition}

\begin{remark}\label{rem:Pell_Lucas_discrepancy}
In the literature, there is sometimes discrepancy on the precise definition of the Pell-Lucas sequence. Though many sources attribute the OEIS sequence \seqnum{A002203} as the ``companion Pell sequence'' (or equivalently, the Pell-Lucas sequence), we choose to follow Koshy~\cite{Koshy2014} and many others whom define the Pell-Lucas sequence as we have done in Definition~\ref{def:Pell_Qell_numbers}, wherein we call $\QellSeq$ the ``associated Pell sequence''.
\end{remark}

\begin{definition}\label{def:balancing_and_Lucas_balancing_numbers}
The \textit{balancing sequence} $\BalSeq$ and \textit{Lucas-balancing sequence} $\LucBalSeq$ are defined by the recurrence relations $B_n = 6 B_{n-1} - B_{n-2}$ and $C_n = 6 C_{n-1} - C_{n-2}$, respectively,
with initial conditions $B_0 = 0$, $B_1 = 1$, $C_0 = 1$, and $C_1 = 3$. In the OEIS, these are sequences \seqnum{A001109} and \seqnum{A001541}~\cite{Sloane-OEIS}.
\end{definition}

\begin{table}[h!]
\renewcommand{\arraystretch}{1.4}
\centering
\begin{tabular}{|c||c|c|c|c|c|c|c|c|c|c|c|c|}
\hline
\blue{$n$} & \redbf{0} & \redbf{1} & \redbf{2} & \redbf{3} & \redbf{4} & \redbf{5} & \redbf{6} & \redbf{7} & \redbf{8} & \redbf{9} & \redbf{10}\\ \hline\hline
\rowcolor{lightgray}
\blue{$F_n$} & 0 & 1 & 1 & 2 & 3 & 5 & 8 & 13 & 21 & 34 & 55\\ \hline
\blue{$L_n$} & 2 & 1 & 3 & 4 & 7 & 11 & 18 & 29 & 47 & 76 & 123\\ \hline\hline
\rowcolor{lightgray}
\blue{$P_n$} & 0 & 1 & 2 & 5 & 12 & 29 & 70 & 169 & 408 & 985 & 2378\\ \hline
\blue{$Q_n$} & 1 & 1 & 3 & 7 & 17 & 41 & 99 & 239 & 577 & 1393 & 3363\\ \hline\hline
\rowcolor{lightgray}
\blue{$B_n$} & 0 & 1 & 6 & 35 & 204 & 1189 & 6930 & 40391 & 235416 & 1372105 & 7997214\\ \hline
\blue{$C_n$} & 1 & 3 & 17 & 99 & 577 & 3363 & 19601 & 114243 & 665857 & 3880899 & 22619537\\ \hline
\end{tabular}
\caption{The first 11 Fibonacci $F_n$, Lucas $L_n$, Pell $P_n$, associated Pell $Q_n$, balancing $B_n$, and Lucas-balancing $C_n$ numbers}
\label{table:Pell_Qell_numbers}
\end{table}

Four of the six sequences in this subsection are $\LucasFormalUSeq$ or $\LucasFormalVSeq$ Lucas sequences. The two that are not Lucas sequences are the associated Pell sequence $\QellSeq$ and the Lucas-balancing sequence $\LucBalSeq$. However, Table~\ref{table:well_known_Lucas_sequences} shows how these two non-Lucas sequences fit into the framework of the $\LucasFormalVSeq$ Lucas sequences (see footnotes).

\begin{table}[!ht]
\renewcommand{\arraystretch}{1.3}
\centering
\begin{tabular}{|c||c|c|c|c|}
\hline
 & \red{name of sequence} & \red{char. poly.} & \red{roots $\alpha$ and $\beta$} & \red{$\Delta$}\\ \hline\hline
\rowcolor{lightgray}
\blue{$\left(U_n(1,-1)\right)_{n \geq 0}$} & Fibonacci  $\FibSeq$ & $x^2 -x - 1$ & $\frac{1 \pm \sqrt{5}}{2}$ & 5\\ \hline
\blue{$\left(V_n(1,-1)\right)_{n \geq 0}$} & Lucas  $\LucSeq$ & $x^2 -x - 1$ & $\frac{1 \pm \sqrt{5}}{2}$ & 5\\ \hline
\rowcolor{lightgray}
\blue{$\left(U_n(2,-1)\right)_{n \geq 0}$} & Pell  $\PellSeq$ & $x^2 -2x - 1$ & $1 \pm \sqrt{2}$ & 8\\ \hline
\blue{$\left(V_n(2,-1)\right)_{n \geq 0}$} & Pell-Lucas\footnotemark $\left(2 Q_n\right)_{n \geq 0}$ & $x^2 -2x - 1$ & $1 \pm \sqrt{2}$ & 8\\ \hline
\rowcolor{lightgray}
\blue{$\left(U_n(6,1)\right)_{n \geq 0}$} & balancing  $\BalSeq$ & $x^2 -6x + 1$ & $3 \pm 2 \sqrt{2}$ & 32\\ \hline
\blue{$\left(V_n(6,1)\right)_{n \geq 0}$} & no name\footnotemark $\left(2 C_n\right)_{n \geq 0}$ & $x^2 -6x + 1$ & $3 \pm 2\sqrt{2}$ & 32\\ \hline
\end{tabular}
\caption{Some well-known Lucas sequences $\LucasFormalUSeq$ and $\LucasFormalVSeq$}
\label{table:well_known_Lucas_sequences}
\end{table}
\footnotetext[1]{Each term in the sequence $\left(V_n(2,-1)\right)_{n \geq 0}$ is twice the value of the corresponding associated Pell term in the sequence $\QellSeq$. Recall Remark~\ref{rem:Pell_Lucas_discrepancy} for naming conventions.}
\footnotetext[2]{Each term in the sequence $\left(V_n(6,1)\right)_{n \geq 0}$ is twice the value of the corresponding Lucas-balancing term in the sequence $\LucBalSeq$.}


\subsection{Sequence statistics}

For the following sequence statistics, let $\GenericSeq$ be any of the sequences given in Subsections~\ref{subsec:Lucas_sequences} and \ref{subsec:particular_sequences}.

\begin{definition}[period]
The \textit{period} of $\GenericSeq$ modulo $m$ is the least integer $r>0$ such that $S_r \equiv S_0 \pmod{m}$ and $S_{r+1} \equiv S_1 \pmod{m}$. Denote this value $r$ by $\pi_S(m)$.
\end{definition}

\begin{definition}[fundamental period] Let $\widetilde{S}_n$ denote the least residue class of $S_n \pmod{m}$. The sequence of residue classes $\widetilde{S}_{n}$ for $0 \leq n \leq \pi_S(m)$ of $S_{n}\pmod{m}$, form the \textit{fundamental period}.
\end{definition}

\begin{definition}[entry point of $m$] The \textit{entry point} of $m$ in $\GenericSeq$ is the least integer $r>0$ (if it exists) such that $m$ divides $S_r$. Denote $r$ by $e_S(m)$. In the literature, this is sometimes called the \textit{rank of appearance} or \textit{restricted period} of $m$.
\end{definition}

\boitegrise{
\begin{convention}[restriction of the moduli $m$ values]\label{conv:gcd_of_m_and_q_equals_1}
\vspace{-.15in}
For the periods of the Lucas sequences $\LucasFormalUSeq$ and $\LucasFormalVSeq$, we consider only the moduli values $m$ such that $\gcd(q,m) = 1$, for otherwise, the sequence may not be purely periodic according to our definition for $\pi_U(m)$ and $\pi_V(m)$ (see Carmichael~\cite[p.~344--345]{Carmichael1920}). Consequently, since $U_0 = 0$, the value $e_U(m)$ is guaranteed to exist whenever $\gcd(q,m) = 1$.
\vspace{-.3in}
\end{convention}}{0.95\textwidth}

\begin{definition}[order of $m$]
Given $\pi_S(m)$ and $e_S(m)$ for $\GenericSeq \pmod{m}$, the value $\frac{\pi_S(m)}{e_S(m)}$ is the \textit{order of $m$} denoted by $\omega_{S}(m)$.
\end{definition}


\subsection{Preliminary Lucas sequence identities}

In this subsection, we collect some (mostly) known Lucas sequence identities that we use to prove our main results in Sections~\ref{sec:theory} and \ref{sec:applications}. Many of the identities are either proven directly in Lucas' 1878 paper~\cite{Lucas1878}, Carmichael's paper from 1913/14~\cite{Carmichael1913}, or they are direct consequences of results in those historic papers. We provide proofs of three identities, \eqref{eq:Lucas_identity_Vorobiev_U_n_like_result}, \eqref{eq:Lucas_identity_Vorobiev_V_n_like_result}, and \eqref{eq:Lucas_identity_for_Ballot_congruence}, which do not appear to have been proven in the literature yet.

\begin{lemma}\label{lem:U_n_Lucas_sequence_identities_part_1}
Set $U_n:=\LucasFormalUTerm$ and $V_n:=\LucasFormalVTerm$. For all $r,s \in \mathbb{Z}$, the following identities hold:
\begin{align}
    U_{2r} &= U_r V_r\label{eq:Lucas_identity_U_2n_equals_U_n_V_n}\\
    U_{2r+s} &= -q^r \cdot U_s + U_{r+s}
    V_r\label{eq:Lucas_identity_13}\\
    U_{-r} &= -\frac{U_r}{q^r} \;\;\text{and}\;\; V_{-r} = \frac{V_r}{q^r}\label{eq:Lucas_identity_negative_index}\\
    V_r &= U_{r+1} - q U_{r-1}\label{eq:Lucas_identity_conversion_of_Vn_to_U_n}\\
    2 U_{r+s} &= U_r V_s + U_s V_r \label{eq:Lucas_identity_adding_indices}\\
    U_{r+s} &= -q \cdot U_{r-1} U_s + U_r U_{s+1} \label{eq:Lucas_identity_Vorobiev_U_n_like_result}\\
    V_{r+s} &= -q \cdot V_{r-1} U_s + V_r U_{s+1} \label{eq:Lucas_identity_Vorobiev_V_n_like_result}\\
    V_{r+s} &= V_r V_s - q^s V_{r-s} \label{eq:Lucas_identity_for_Ballot_congruence}\\
    V_r &= -2 q \cdot U_{r-1} + p \cdot U_r.\label{eq:Lucas_identity_V_r_in_terms_of_U_r_and_U_r_minus_1}
\end{align}
Moreover for all $r,s \geq 0$, the following identities hold:
\begin{align}
    r \text{ divides } s &\Rightarrow U_r \text{ divides } U_s\label{eq:Lucas_identity_U_r_divides_U_s}\\
    \gcd(U_r, V_r) &= 1 \text{ or } \,2.\label{eq:Lucas_identity_GCD_of_Un_and_Vn}
\end{align}
\end{lemma}

\begin{proof}[Proof of Identities~\eqref{eq:Lucas_identity_U_2n_equals_U_n_V_n} to \eqref{eq:Lucas_identity_adding_indices}]
Identities~\eqref{eq:Lucas_identity_U_2n_equals_U_n_V_n}, \eqref{eq:Lucas_identity_13}, and \eqref{eq:Lucas_identity_negative_index}, respectively, follow from Lucas~\cite[Equations~(3), (13), and (50)]{Lucas1878}. Identity~\eqref{eq:Lucas_identity_conversion_of_Vn_to_U_n} is readily verified by the Binet formulas for $U_n$ and $V_n$ and the fact that $q=\alpha\beta$. Identity~\eqref{eq:Lucas_identity_adding_indices} follows from Lucas~\cite[Equation~(49)]{Lucas1878}.
\end{proof}

\begin{proof}[Proof of Identity~\eqref{eq:Lucas_identity_Vorobiev_U_n_like_result}]
Observe that $U_{r+s} = U_{(r-1)+(s+1)}$, and hence we have
\begin{align*}
    2 U_{r+s} &= U_{r-1} V_{s+1} + U_{s+1} V_{r-1} &\text{by Identity~\eqref{eq:Lucas_identity_adding_indices}}\\
    &= U_{r-1} (U_{s+2} - q U_s) + U_{s+1} (U_r - q U_{r-2}) &\text{by Identity~\eqref{eq:Lucas_identity_conversion_of_Vn_to_U_n}}\\
    &= (-q U_{r-1} U_s + U_r U_{s+1}) + (-q U_{r-2} U_{s+1} + U_{r-1} U_{s+2}).
\end{align*}
It suffices to show that the left and right parenthesized summands in the third equality are equal. To that end, observe that $-qU_{r-2} = U_r - p U_{r-1}$ and $U_{s+2} = p U_{s+1} - q U_s$ by the recurrence relation for $\LucasUSeq$.  Thus the right parenthesized summand becomes
$$ U_r U_{s+1} - p U_{r-1} U_{s+1} + p U_{r-1} U_{s+1} - q U_{r-1} U_s,$$
which is exactly the left parenthesized summand upon cancellation of the middle two summands above. Hence we have $2 U_{r+s} = 2 (-q U_{r-1} U_s + U_r U_{s+1})$, and the result follows.
\end{proof}

\begin{proof}[Proof of Identity~\eqref{eq:Lucas_identity_Vorobiev_V_n_like_result}]
By the Binet forms for $U_n$ and $V_n$ and the fact that $q = \alpha \beta$, we have the following sequence of equalities:
\begin{align*}
-q V_{r-1} U_s &+ V_r U_{s+1}\\
    &= -\alpha\beta \cdot (\alpha^{r-1} + \beta^{r-1}) \cdot \frac{\alpha^s - \beta^s}{\alpha - \beta} + (\alpha^r + \beta^r) \cdot \frac{\alpha^{s+1} - \beta^{s+1}}{\alpha - \beta}\\
    &= \frac{1}{\alpha - \beta} \Bigl[ (-\alpha\beta)(\alpha^{r+s-1} - \beta^{r+s-1} - \alpha^{r-1}\beta^s + \alpha^s\beta^{r-1})\\
    &\hspace{1in} +  (\alpha^{r+s+1} - \beta^{r+s+1} - \alpha^r \beta^{s+1} + \alpha^{s+1} \beta^r) \Bigr]\\
    &= \frac{1}{\alpha - \beta} (-\alpha^{r+s} \beta + \alpha \beta^{r+s} + \alpha^{r+s} \alpha - \beta^{r+s} \beta)\\
    &= \frac{1}{\alpha - \beta} \bigl(\alpha^{r+s} (\alpha - \beta) + \beta^{r+s} (\alpha - \beta) \bigr)\\
    &= \alpha^{r+s} + \beta^{r+s}\\
    &= V_{r+s},
\end{align*}
where the third equality follows by distributing the $-\alpha\beta$ into the first parenthesized summand of the second equality, thereby canceling the $-\alpha^r \beta^{s+1}$ and $\alpha^{s+1} \beta^r$ in the second parenthesized summand of the second equality.
\end{proof}

\begin{proof}[Proof of Identity~\eqref{eq:Lucas_identity_for_Ballot_congruence}]
By the Binet form for $V_n$ and the fact that $q = \alpha \beta$, we have the following sequence of equalities:
\begin{align*}
V_r V_s - q^s V_{r-s} &= (\alpha^r + \beta^r) (\alpha^s + \beta^s) - (\alpha\beta)^s (\alpha^{r-s} + \beta^{r-s})\\
    &= \left( \alpha^{r+s} + \beta^{r+s} + \alpha^r \beta^s + \alpha^s \beta^r \right) - \left( \alpha^r \beta^s + \alpha^s \beta^r \right)\\
    &= \alpha^{r+s} + \beta^{r+s}\\
    &= V_{r+s},
\end{align*}
as desired.
\end{proof}

\begin{proof}[Proof of Identity~\eqref{eq:Lucas_identity_V_r_in_terms_of_U_r_and_U_r_minus_1}]
Observe that $V_r = V_{1 + (r-1)} = -q \cdot V_0 U_{r-1} + V_1 U_r$ by Identity~\eqref{eq:Lucas_identity_Vorobiev_V_n_like_result}, and since $V_0 = 2$ and $V_1 = p$, we have $V_r = -2q \cdot U_{r-1} + p \cdot U_r$, as desired.
\end{proof}

\begin{proof}[Proof of Identity~\eqref{eq:Lucas_identity_U_r_divides_U_s}]
This follows from Carmichael~\cite[Theorem~IV]{Carmichael1913}.
\end{proof}

\begin{proof}[Proof of  Identity~\eqref{eq:Lucas_identity_GCD_of_Un_and_Vn}]
Lucas perhaps inadvertently states that $\gcd(U_r, V_r) = 1$~\cite[p.~200]{Lucas1878}; however, the correct statement is proven by Carmichael~\cite[Theorem~II]{Carmichael1913}.
\end{proof}

For $\frac{r}{d}$ and $\frac{s}{d}$ both odd where $d = \gcd(r,s)$, the fact that $\gcd(V_r, V_s) = V_d$ has been a known result since 1913 by Carmichael~\cite[Theorem~VII]{Carmichael1913}. It was not until 1991 when McDaniel proved the following result which handles the case when $r$ and $s$ are not divisible by the same power of 2. We use this result to prove that the GCD of consecutive terms in $\LucasFormalVSeq$ is 1 or 2, a result we use in Corollary~\ref{cor:Sufficiency_condition_for_equality_of_periods}.

\begin{lemma}[{McDaniel~\cite[Main Theorem~(ii)]{McDaniel1991}}]\label{lem:McDaniels_GCD_result}
Let $r = 2^a r'$ and $s = 2^b s'$, with $r'$ and $s'$ odd, $a,b \geq 0$, and set $d := \gcd(r,s)$. Then
$$\gcd(V_r, V_s) =
\begin{cases}
    V_d, \text{ if $a = b$};\\
    1 \text{ or } 2, \text{ if $a \neq b$}.
\end{cases}
$$
\end{lemma}

\begin{corollary}\label{cor:GCD_of_consecutive_terms_in_Vn_is_1_or_2}
Set $V_n:=\LucasFormalVTerm$. Then for all $k \geq 0$, it follows that
$$\gcd(V_k, V_{k+1}) =
\begin{cases}
    1, \text{ if $p$ is odd};\\
    2, \text{ if $p$ is even}.
\end{cases}
$$
\end{corollary}

\begin{proof}
Let $k \geq 0$ be given. Since $k$ and $k+1$ have different parities, we know that $\gcd(V_k, V_{k+1}) = 1 \text{ or } 2$ by Lemma~\ref{lem:McDaniels_GCD_result}. By the recurrence for $\LucasVSeq$ and the fact that $V_0 = 2$ and $V_1 = p$, it is readily verified that
\begin{align}
    p \text{ odd and } q \text{ odd} &\Longrightarrow \left( V_n \text{ even} \Leftrightarrow \text{3 divides $n$}\right)\text{ for all $n \geq 0$}\label{eq:GCD_when_p_odd_and_q_odd}\\
    p \text{ odd and } q \text{ even} &\Longrightarrow V_n \text{ odd for all $n\geq 1$}\label{eq:GCD_when_p_odd_and_q_even}\\
    p \text{ even and } q \text{ odd} &\Longrightarrow V_n \text{ even for all $n \geq 0$}.\label{eq:GCD_when_p_even_and_q_even}
\end{align}
If $p$ and $q$ are odd, then consecutive terms are either both odd or have different parities by Implication~\eqref{eq:GCD_when_p_odd_and_q_odd}, and thus $\gcd(V_k, V_{k+1}) = 1$. Moreover, if $p$ is odd and $q$ is even, then $\gcd(V_k, V_{k+1}) = 1$ by Implication~\eqref{eq:GCD_when_p_odd_and_q_even}. Finally, if $p$ is even, then $q$ is necessarily odd since $\gcd(p,q)=1$, and hence $\gcd(V_k, V_{k+1}) = 2$ by Implication~\eqref{eq:GCD_when_p_even_and_q_even}. Thus the result holds.
\end{proof}

In the next lemma, we generalize the following well-known Fibonacci/Lucas identity: $F_{2n+1} = (-1)^n + F_n L_{n+1} = (-1)^{n+1} + F_{n+1} L_n$. This will be used in proving Lemmas~\ref{lem:Desmond_lemma_1_generalization} and \ref{lem:Desmond_lemma_2_generalization} and many results in Section~\ref{sec:theory}.

\begin{lemma}\label{lem:odd_subscript_Un_result}
Set $U_n:=\LucasFormalUTerm$  and $V_n:=\LucasFormalVTerm$. Then the following identities hold for all $n \in \mathbb{Z}$:
\begin{align}
    U_{2n+1} &= q^n + U_n V_{n+1}\label{eq:Lucas_identity_odd_subscript_1}\\
    U_{2n+1} &= -q^n + U_{n+1} V_n.\label{eq:Lucas_identity_odd_subscript_2}
\end{align}
\end{lemma}

\begin{proof}[Proof of Identity~\eqref{eq:Lucas_identity_odd_subscript_1}]
If we set $r:=n+1$ and $s:=-1$ in Identity~\eqref{eq:Lucas_identity_13} of Lemma~\ref{lem:U_n_Lucas_sequence_identities_part_1}, then
\begin{align*}
    U_{2n+1} = U_{2(n+1)+(-1)} &= -q^{n+1} \cdot U_{-1} + U_{(n+1)-1} V_{n+1}\\
    &= -q^{n+1} \cdot \left(-\frac{U_1}{q}\right) + U_n V_{n+1} &\text{by  Identity~\eqref{eq:Lucas_identity_negative_index} of Lemma~\ref{lem:U_n_Lucas_sequence_identities_part_1}}\\
    &= q^n + U_n V_{n+1},
\end{align*}
since $U_1 = 1$. Thus the result follows.
\end{proof}

\begin{proof}[Proof of Identity~\eqref{eq:Lucas_identity_odd_subscript_2}]
If we set $r:=n$ and $s:=1$ in Identity~\eqref{eq:Lucas_identity_13} of Lemma~\ref{lem:U_n_Lucas_sequence_identities_part_1}, then
\begin{align*}
    U_{2n+1} &= -q^n \cdot U_1 + U_{n+1} V_n = -q^n + U_{n+1} V_n,
\end{align*}
since $U_1 = 1$. Thus the result follows.
\end{proof}

The next two Lemmas~\ref{lem:Desmond_lemma_1_generalization} and \ref{lem:Desmond_lemma_2_generalization} give conditions for the divisibility of Lucas sequence terms $\LucasFormalUTerm$ or $\LucasFormalVTerm$ by an integer $m \geq 1$. We generalize the following two results given by Desmond in the Fibonacci/Lucas setting~\cite[Lemmas~1 and 2]{Desmond1978}:
\begin{align*}
F_{2n} \equiv 0 \pmodd{m} \text{ and } F_{2n+1} \equiv (-1)^n \pmodd{m} &\;\Longleftrightarrow\; F_n \equiv 0 \pmodd{m}, \text{and}\\
F_{2n} \equiv 0 \pmodd{m} \text{ and } F_{2n+1} \equiv (-1)^{n+1} \pmodd{m} &\;\Longleftrightarrow\; L_n \equiv 0 \pmodd{m}.
\end{align*}
Interestingly enough, these two results generalize to all Lucas sequences $\LucasFormalUSeq$ and $\LucasFormalVSeq$ in the manner given in the next two lemmas. Despite their use in results in Subsections~\ref{subsec:entry_points_and_order} and \ref{subsec:consequences_of_existence_of_entry_point_in_V_n}, the following two lemmas are interesting in their own right. But first we give an important remark justifying why the following two lemmas and many results in Sections~\ref{sec:theory} and \ref{sec:applications} hold for all $n \in \mathbb{Z}$ and not just $n \geq 0$.

\boitegrise{
\begin{remark}\label{rem:negative_index_terms_under_a_modulus}
\vspace{-.15in}
        Recall that by Identity~\eqref{eq:Lucas_identity_negative_index} of Lemma~\ref{lem:U_n_Lucas_sequence_identities_part_1}, we have $U_{-r} = -\frac{U_r}{q^r}$ and $V_{-r} = \frac{V_r}{q^r}$. For $r > 0$, the apparent fractional values $-\frac{U_r}{q^r}$ and $\frac{V_r}{q^r}$ modulo $m$ are well defined since $\gcd(q,m) = 1$ by Convention~\ref{conv:gcd_of_m_and_q_equals_1}, and thus the inverse of $q$ modulo $m$ exists. Therefore, the values $-U_r \cdot q^{-r}$ and $V_r \cdot q^{-r}$ are indeed integers modulo $m$.
\vspace{-.3in}
\end{remark}}{0.95\textwidth}

\begin{lemma}\label{lem:Desmond_lemma_1_generalization}
Set $U_n:=\LucasFormalUTerm$ and $V_n:=\LucasFormalVTerm$. Then for all $m \geq 1$ and $n \in \mathbb{Z}$, the following implications hold:
\begin{align}
    U_{2n} \equiv 0 \pmodd{m} \text{ and } U_{2n+1} \equiv q^n \pmodd{m} &\;\Longleftarrow\; U_n \equiv 0 \pmodd{m} \label{eq:Desmond_lemma_1_right_to_left}\\
    U_{2n} \equiv 0 \pmodd{m} \text{ and } U_{2n+1} \equiv q^n \pmodd{m} &\;\Longrightarrow\; U_n \equiv 0 \pmodd{m} \label{eq:Desmond_lemma_1_left_to_right},
\end{align}
where the second implication holds if $p$ or $m$ is odd.
\end{lemma}

\begin{proof}
Let $m \geq 1$ and $n \in \mathbb{Z}$ be given, and assume that $U_n \equiv 0 \pmod{m}$. Since $U_{2n} = U_n V_n$ by Identity~\eqref{eq:Lucas_identity_U_2n_equals_U_n_V_n} of Lemma~\ref{lem:U_n_Lucas_sequence_identities_part_1}, it follows that $U_{2n} \equiv 0 \pmod{m}$. Moreover, we have $U_{2n+1} \equiv q^n + U_n V_{n+1} \equiv q^n \pmod{m}$ by Identity~\eqref{eq:Lucas_identity_odd_subscript_1} of Lemma~\ref{lem:odd_subscript_Un_result} since $U_n \equiv 0 \pmod{m}$. We conclude that Implication~\eqref{eq:Desmond_lemma_1_right_to_left} follows.

Now assume that $U_{2n} \equiv 0 \pmod{m}$ and $U_{2n+1} \equiv q^n \pmod{m}$. To prove Implication~\eqref{eq:Desmond_lemma_1_left_to_right}, we suppose that $p$ or $m$ is odd. However, for all $p$ and $m$ regardless of parity, we claim that $U_n V_{n+r} \equiv 0 \pmod{m}$ for all $r \in \mathbb{Z}$. To that end, we first show that
$$U_n V_{n+r} \equiv 0 \equiv U_n V_{n+(r+1)} \pmodd{m} \;\implies\; U_n V_{n+(r+2)} \equiv 0 \pmodd{m}$$
for all $r \geq 0$. We proceed by induction on $r$. Since $U_{2n} = U_n V_n$ by Identity~\eqref{eq:Lucas_identity_U_2n_equals_U_n_V_n} of Lemma~\ref{lem:U_n_Lucas_sequence_identities_part_1} and $U_{2n} \equiv 0 \pmod{m}$ by assumption, we have $U_n V_n \equiv 0 \pmod{m}$. Moreover, since $U_{2n + 1} = q^n + U_n V_{n+1}$ by Identity~\eqref{eq:Lucas_identity_odd_subscript_1} of Lemma~\ref{lem:odd_subscript_Un_result} and $U_{2n + 1} \equiv q^n \pmod{m}$ by assumption, we have $q^n \equiv q^n + U_n V_{n+1} \pmod{m}$ and hence $U_n V_{n+1} \equiv 0 \pmod{m}$. Thus $U_n V_n \equiv 0 \equiv U_n V_{n+1}\pmod{m}$ holds. Observe that
$$U_n V_{n+2} = U_n \cdot (p V_{n+1} - q V_n) = p U_n V_{n+1} - q U_n V_n \equiv 0 \pmodd{m},$$
and hence the base case holds. Now suppose that $U_n V_{n+k} \equiv 0 \equiv U_n V_{n+(k+1)}\pmod{m}$ for some $k \geq 0$, and observe that
$$U_n V_{n+(k+2)} = U_n (p V_{n+(k+1)} - q V_{n+k})  = p U_n V_{n+(k+1)} - q U_n V_{n+k} \equiv 0 \pmodd{m},$$
concluding the inductive step, and hence $U_n V_{n+r} \equiv 0 \pmod{m}$ for all $r \geq 0$. To prove that $U_n V_{n+r} \equiv 0 \pmod{m}$ also holds for all $r \in \mathbb{Z}$, we give an analogous inductive argument  by showing $U_n V_{n+(r+1)} \equiv 0 \equiv U_n V_r \pmod{m}$ implies $ U_n V_{n+(r-1)} \equiv 0 \pmod{m}$ for all $r \leq 0$. We have already shown that $U_n V_{n+1} \equiv 0 \equiv U_n V_n \pmod{m}$, so we show $U_n V_{n-1} \equiv 0 \pmod{m}$ is forced. By the recurrence relation for $\LucasVSeq$, we have $q V_{n-1} = p V_n - V_{n+1}$ and thus $V_{n-1} \equiv pq^{-1} V_n - q^{-1} V_{n+1} \pmod{m}$ since $\gcd(q,m) = 1$. Observe that
$$U_n V_{n-1} = U_n \cdot (pq^{-1} V_n - q^{-1} V_{n+1}) = pq^{-1} U_n V_n - q^{-1} U_n V_{n+1} \equiv 0 \pmodd{m},$$
and hence the base case holds. Now suppose that $U_n V_{n+(k+1)} \equiv 0 \equiv U_n V_{n+k}\pmod{m}$ for some $k \leq 0$, and observe that
$$U_n V_{n+(k-1)} = U_n (pq^{-1} V_{n+k} - q^{-1} V_{n+(k+1)})  = pq^{-1} U_n V_{n+k} - q^{-1} U_n V_{n+(k+1)} \equiv 0 \pmodd{m},$$
concluding the inductive step, and hence $U_n V_{n+r} \equiv 0 \pmod{m}$ for all $r \leq 0$. We conclude that $U_n V_{n+r} \equiv 0 \pmod{m}$ holds for all $r \in \mathbb{Z}$, and for all values $p$ and $m$ regardless of parity. Now suppose that $p$ or $m$ is odd.

\medskip

\noindent \textbf{(CASE 1):} Suppose that $p$ is odd. Then $V_1 = p = 2j+1$ for some $j \in \mathbb{Z}$. Observe that $U_n V_0 \equiv 0 \equiv U_n V_1 \pmod{m}$ since $U_n V_{n+r} \equiv 0 \pmod{m}$ for all $r \in \mathbb{Z}$, and thus
\begin{align*}
    2 U_n \equiv 0 \equiv (2j+1) \cdot U_n \pmodd{m} &\implies (2j+1) \cdot U_n - 2 U_n \equiv 0 \pmodd{m}\\
    &\implies (j-1) \cdot 2 U_n + U_n \equiv 0 \pmodd{m},
\end{align*}
which implies that $U_n \equiv 0 \pmod{m}$, as desired, since $2 U_n \equiv 0 \pmod{m}$.

\medskip

\noindent \textbf{(CASE 2):} Suppose that $m$ is odd. Then $U_n V_0 \equiv 0 \pmod{m}$ implies that $U_n \equiv 0 \pmod{m}$, as desired, since $\gcd(V_0,m) = 1$.

\medskip

\noindent We conclude that Implication~\eqref{eq:Desmond_lemma_1_left_to_right} follows.
\end{proof}

\begin{lemma}\label{lem:Desmond_lemma_2_generalization}
Set $U_n:=\LucasFormalUTerm$ and $V_n:=\LucasFormalVTerm$. Then for all $m \geq 1$ and $n \in \mathbb{Z}$, the following biconditional holds:
$$ U_{2n} \equiv 0 \pmodd{m} \text{ and } U_{2n+1} \equiv -q^n \pmodd{m} \;\Longleftrightarrow\; V_n \equiv 0 \pmodd{m}.$$
\end{lemma}

\begin{proof}
Let $m \geq 1$ be given, and assume that $U_{2n} \equiv 0 \pmod{m}$ and $U_{2n+1} \equiv -q^n \pmod{m}$. We claim that $U_{n+r} V_n \equiv 0 \pmod{m}$ for all $r \geq 0$. To that end, we first show that
$$U_{n+r} V_n \equiv 0 \equiv U_{n+(r+1)} V_n \!\!\!\pmod{m} \;\implies\; U_{n+(r+2)} V_n \equiv 0 \!\!\!\pmod{m}$$
for all $r \geq 0$. We proceed by induction on $r$. Since $U_{2n} = U_n V_n$ by Identity~\eqref{eq:Lucas_identity_U_2n_equals_U_n_V_n} of Lemma~\ref{lem:U_n_Lucas_sequence_identities_part_1} and $U_{2n} \equiv 0 \pmod{m}$ by assumption, we have $U_n V_n \equiv 0 \pmod{m}$. Moreover, since $U_{2n + 1} = -q^n + U_{n+1} V_n$ by Identity~\eqref{eq:Lucas_identity_odd_subscript_2} of Lemma~\ref{lem:odd_subscript_Un_result} and $U_{2n + 1} \equiv -q^n \pmod{m}$ by assumption, we have $-q^n \equiv -q^n + U_{n+1} V_n \pmod{m}$ and hence $U_{n+1} V_n \equiv 0 \pmod{m}$. Thus $U_n V_n \equiv 0 \equiv U_{n+1} V_n\pmod{m}$ holds. Observe that
$$U_{n+2} V_n = (p U_{n+1} - q U_n) \cdot V_n = p U_{n+1} V_n - q U_n V_n \equiv 0 \pmodd{m},$$
and hence the base case holds. Now suppose that $U_{n+k} V_n \equiv 0 \equiv U_{n+(k+1)} V_n\pmod{m}$ for some $k \geq 0$, and observe that
$$U_{n+(k+2)} V_n = (p U_{n+(k+1)} - q U_{n+k}) \cdot V_n = p U_{n+(k+1)} V_n - q U_{n+k} V_n \equiv 0 \pmodd{m},$$
concluding the inductive step, and hence $U_{n+r} V_n \equiv 0 \pmod{m}$ for all $r \geq 0$. Thus there exists an $r_0 \geq 0$ such that $n + r_0 = s \cdot \pi_U(m) + 1$ for some $s \in \mathbb{Z}$, and so we have
$$V_n \equiv U_1 V_n \equiv U_{s \cdot \pi_U(m) + 1} V_n \equiv U_{n + r_0} V_n \equiv 0 \pmod{m},$$
as desired. This proves the sufficiency condition.

Now assume that $V_n \equiv 0 \pmod{m}$. Since $U_{2n} = U_n V_n$ by Identity~\eqref{eq:Lucas_identity_U_2n_equals_U_n_V_n} of Lemma~\ref{lem:U_n_Lucas_sequence_identities_part_1}, we have $U_{2n} \equiv 0 \pmod{m}$. Moreover, $U_{2n+1} \equiv -q^n + U_{n+1} V_n \equiv -q^n \pmod{m}$ by Identity~\eqref{eq:Lucas_identity_odd_subscript_2} of Lemma~\ref{lem:odd_subscript_Un_result} since $V_n \equiv 0 \pmod{m}$. This proves the necessity condition.
\end{proof}

\begin{example}
We consider Lemma~\ref{lem:Desmond_lemma_2_generalization} with negatively-indexed terms to provide clarity on Remark~\ref{rem:negative_index_terms_under_a_modulus} and to show the utility of this lemma. Set $(p,q,m,n) := (3,4,9,-3)$. It is readily verified that $V_3 = -9$, and so $V_{-3} = \frac{V_3}{q^3} = \frac{-9}{4^3} \equiv -9 \cdot 7^3 \pmod{9}$, where the first equality holds by Identity~\eqref{eq:Lucas_identity_negative_index} of Lemma~\ref{lem:U_n_Lucas_sequence_identities_part_1}, and the congruence holds since 7 is the multiplicative inverse of 4 modulo 9. Thus $V_{-3} \equiv 0 \pmod{9}$.

By Lemma~\ref{lem:Desmond_lemma_2_generalization}, this must force $U_{2n} \equiv 0 \pmod{m}$ and $U_{2n+1} \equiv -q^n \pmod{m}$; that is, $U_{-6} \equiv 0 \pmod{9}$ and $U_{-5} \equiv -4^{-3} \pmod{9}$ must hold. We verify this. Since $U_{-6} = U_{-3} V_{-3}$ by Identity~\eqref{eq:Lucas_identity_U_2n_equals_U_n_V_n} of Lemma~\ref{lem:U_n_Lucas_sequence_identities_part_1} and $V_{-3} \equiv 0 \pmod{9}$, we have $U_{-6} \equiv 0 \pmod{9}$ as desired. Moreover, it is readily verified that $U_{5} = -11$, and so $U_{-5} = - \frac{U_5}{q^5} = -\frac{-11}{4^5} \equiv 11 \cdot 7^5 \pmod{9}$, again by Identity~\eqref{eq:Lucas_identity_negative_index} and the inverse of 4 modulo 9. Since $11 \cdot 7^5 \equiv 8 \pmod{9}$, it follows that $U_{-5} \equiv 8 \pmod{9}$. Lastly, observe that $-4^{-3} \equiv -7^3 \equiv 8 \pmod{9}$ also holds, and thus $U_{-5} \equiv -4^{-3} \pmod{9}$ as desired.
\end{example}


\section{Theory: statistics in the Lucas sequences}\label{sec:theory}
In this section, we provide a number of interesting results regarding the theory of the statistics of period, entry point, and order in the Lucas sequences. We remind the reader of Convention~\ref{conv:gcd_of_m_and_q_equals_1}; that is, we are considering only moduli values $m$ such that $\gcd(q,m) = 1$. Since many of our results are for Lucas sequences with $q = \pm 1$, this $\gcd$-restriction does not limit any values $m \geq 1$.

\boitegrise{
\begin{convention}\label{conv:abbreviations_for_symbols}
\vspace{-.15in}
Throughout the rest of this paper, for brevity in the proofs, the statitics of period, entry point, and order of $m$, respectively, will be denoted as follows:
\vspace{-.3in}
\begin{multicols}{3}
\begin{enumerate}[(i)]
    \item[] $\pi_U := \pi_U(m)$ 
    \item[] $\pi_V := \pi_V(m)$
    \item[] $e_U := e_U(m)$
    \item[] $e_V := e_V(m)$
    \item[] $\omega_U := \omega_U(m)$
    \item[] $\omega_V := \omega_V(m)$
\end{enumerate}
\end{multicols}
\vspace{-.45in}
\end{convention}}{0.95\textwidth}

\begin{remark}
In this section, results already known are denoted by the term ``proposition'', and we reserve the term ``theorem'' for statements and proofs that are our own.
\end{remark}

\begin{remark}\label{rem:when_m_equals_1_or_2_regarding_statistics}
In many of our main results, we consider only moduli values $m > 2$. This is because when $m=1$, we trivially have $\pi_U(m) = e_U(m) = 1$ and $\pi_V(m) = e_V(m) = 1$. When $m=2$ and $q = \pm 1$, by the recurrences for $\LucasFormalUSeq$ and $\LucasFormalVSeq$, respectively, and elementary parity arguments, it is readily verified that
$$\pi_U(2) = e_U(2) =
\begin{cases}
    2, \text{ if $p$ is even};\\
    3, \text{ if $p$ is odd};
\end{cases}
\text{ and }\;\;
\pi_V(2) = e_V(2) =
\begin{cases}
    1, \text{ if $p$ is even};\\
    3, \text{ if $p$ is odd}.
\end{cases}
$$
Hence when $m=1,2$, we have $\omega_U(m) = \omega_V(m) = 1$.
\end{remark}

\boitegrise{
\begin{convention}[the multiplier]\label{conv:the_multiplier}
\vspace{-.15in}
None of our arguments in this paper use the fact that $\omega_U(m)$ is the multiplicative order modulo $m$ of the so-called \textit{multiplier}, the value $U_{e_U(m) + 1} \pmod{m}$ after the first zero $U_{e_U(m)} \pmod{m}$ in the period of $\LucasFormalUSeq$~\cite[p.~372]{Renault2013}. We intentionally take this approach since our arguments can be done without exploiting this fact. Additionally, this fact does not hold for half of the sequences considered in this paper, namely the $\LucasFormalVSeq$. For example, $\LucSeq$ modulo 11 has $\pi_L(11) = 10$, $e_L(11)=5$, and hence $\omega_L(11) = 2$. The value after this first zero is $L_{e_L(11)+1} = L_6 \equiv 7 \pmod{11}$; however, 7 has order 10 modulo 11, and $\omega_L(11) = 2 \neq 10$.
\vspace{-.5in}
\end{convention}}{0.95\textwidth}


\subsection{Equally spaced zeros in \texorpdfstring{$\LucasFormalUSeq$}{Un} modulo \texorpdfstring{$m$}{m}}

In 1960, Wall proved that the zeros in the Fibonacci sequence modulo $m$ form a simple arithmetic progression~\cite[Theorem~3]{Wall1960}. We generalize his proof to the setting of all Lucas sequences $\LucasFormalUSeq$. Our proof of the necessity condition is substantially shorter than Wall's proof, as we employ an identity proven by Carmichael in 1913 (Identity~\eqref{eq:Lucas_identity_U_r_divides_U_s} of Lemma~\ref{lem:U_n_Lucas_sequence_identities_part_1}) and an identity proven by Lucas in 1878 (Identity~\eqref{eq:Lucas_identity_negative_index} of Lemma~\ref{lem:U_n_Lucas_sequence_identities_part_1}).

\begin{theorem}\label{thm:equally_spaced_zeros_in_U_n_sequences}
Let $m \geq 2$ be given, and set $U_n:=\LucasFormalUSeq$. Then for all $k \in \mathbb{Z}$, we have
$$U_n \equiv 0 \pmodd{m} \;\Longleftrightarrow\; n = ke_U(m).$$
That is, the zeros in $\LucasUSeq$ modulo $m$ are equally spaced.
\end{theorem}

\begin{proof}
Let $m \geq 2$ be given. Observe that $U_{e_U} \equiv 0 \pmod{m}$, and so by Identity~\eqref{eq:Lucas_identity_U_r_divides_U_s} of Lemma~\ref{lem:U_n_Lucas_sequence_identities_part_1}, we have $U_{ke_U} \equiv 0 \pmod{m}$ for all $k \geq 0$. Moreover since $U_{-r} = -q^{-r} \cdot U_r$ for all $r \in \mathbb{Z}$ by Identity~\eqref{eq:Lucas_identity_negative_index} of Lemma~\ref{lem:U_n_Lucas_sequence_identities_part_1}, then $U_{ke_U} \equiv 0 \pmod{m}$ for all $k \in \mathbb{Z}$. This proves the necessity condition.

To prove the sufficiency condition, suppose by way of contradiction that there exists an index $j>0$ such that $U_j \equiv 0 \pmod{m}$ with $ke_U < j < (k+1) e_U$ for some $k \geq 1$. Observe that
\begin{align*}
    ke_U < j < (k+1) e_U &\implies ke_U - ke_U < j - ke_U < (k+1) e_U - ke_U\\
    &\implies 0 < j-ke_U < e_U.
\end{align*}
Since $U_j \equiv 0 \pmod{m}$ and $U_{ke_U} \equiv 0 \pmod{m}$. It follows that
\begin{align*}
    U_{j-ke_U} = U_{j+(-ke_U)} &= -q \cdot U_{j-1} U_{-ke_U} + U_j U_{-ke_U+1} \\
    &= -q \cdot U_{j-1} \left( - q^{-ke_U} \cdot U_{ke_U} \right) + U_j U_{-ke_U+1}\\
    &\equiv 0 \pmodd{m},
\end{align*}
where the first two equalities hold by Identities~\eqref{eq:Lucas_identity_Vorobiev_U_n_like_result} and \eqref{eq:Lucas_identity_negative_index}, respectively, of Lemma~\ref{lem:U_n_Lucas_sequence_identities_part_1}, and the congruence holds since $U_j \equiv 0 \pmod{m}$ and $U_{ke_U} \equiv 0 \pmod{m}$. Thus we have $U_{j-ke_U} \equiv 0 \pmod{m}$, contradicting the fact that $e_U$ is the entry point of $m$. We conclude that $U_n \equiv 0 \pmod{m}$ if and only if $n = ke_U$ for $k \in \mathbb{Z}$.
\end{proof}

\begin{corollary}\label{cor:entry_point_divides_period_in_Un_sequence}
Let $m \geq 2$ be given and set $U_n:=\LucasFormalUSeq$. Then $e_U(m)$ divides $\pi_U(m)$, and in particular, we have $\pi_U(m) = k e_U(m)$ for some $k \geq 1$.
\end{corollary}

\begin{proof}
Let $m \geq 2$ be given. Observe that $U_{\pi_U} \equiv 0 \pmod{m}$. Then the result follows from Theorem~\ref{thm:equally_spaced_zeros_in_U_n_sequences} since $U_{\pi_U} \equiv 0 \pmod{m}$ implies that $\pi_U = k e_U$ for some $k \in \mathbb{Z}$. Moreover, since $e_U \leq \pi_U$ necessarily holds, we know that $k \geq 1$.
\end{proof}


\subsection{Period: the statistics \texorpdfstring{$\pi_U(m)$}{period of m} and \texorpdfstring{$\pi_V(m)$}{period of m} and equality condition}

In the Fibonacci setting, it is well known that the period $\pi_F(m)$ of the sequence $\FibSeq$ modulo $m$ is even for all $m > 2$ (see Wall~\cite[Theorem~4]{Wall1960}). We give a sufficiency criteria for Lucas sequences $\LucasFormalUSeq$ to also have even periods for $m > 2$. This follows from the following result by Renault~\cite{Renault2013}.

\begin{proposition}[{\cite[p.~373]{Renault2013}}]\label{prop:Renault_order_of_q_divides_period}
Let $m \geq 1$ be given. Then $\ord_m(q)$ divides $\pi_U(m)$, where $\ord_m(q)$ denotes the multiplicative order of $q$ modulo $m$.
\end{proposition}

\begin{corollary}\label{cor:Period_in_Un_is_even_when_q_equals_minus_1}
Let $m > 2$ be given and set $U_n:=\LucasFormalUTerm$. If $q = -1$, then $\pi_U(m)$ is even.
\end{corollary}

\begin{proof}
Since $q = -1$ and $\ord_m(-1) = 2$ whenever $m > 2$, the result follows from Proposition~\ref{prop:Renault_order_of_q_divides_period}.
\end{proof}

\begin{remark}\label{rem:even_periods}
While the periods are guaranteed to be even when $q = -1$, this does not necessarily hold when $q = 1$. For example for $(p,q):=(6,1)$, the sequence $\LucasFormalUSeq$ is the balancing sequence $\BalSeq$. It is readily verified that $\pi_B(7)=3$.
\end{remark}

In the following theorem, we reveal a divisibility relationship between the periods of $\LucasFormalUSeq$ and $\LucasFormalVSeq$ modulo $m$.

\begin{theorem}\label{thm:period_of_V_n_divides_period_of_U_n}
Let $m \geq 1$ be given and set $U_n:=\LucasFormalUTerm$ and $V_n:=\LucasFormalVTerm$. Then $\pi_V(m)$ divides $\pi_U(m)$. 
\end{theorem}

\begin{proof}
Let $m \geq 1$ be given. Since we know that $V_0 = 2$ and $V_1 = p$, it suffices to show that $V_{\pi_U} \equiv 2 \pmod{m}$ and $V_{\pi_U + 1} \equiv p \pmod{m}$. First observe that by the recurrence relation for $\LucasUSeq$, we have $U_{\pi_U + 1} = p \cdot U_{\pi_U} - q \cdot U_{\pi_U - 1}$, and hence $1 \equiv - q \cdot U_{\pi_U - 1} \pmod{m}$ since $U_{\pi_U + 1} \equiv 1 \pmod{m}$. But $\gcd(q,m)=1$ by Convention~\ref{conv:gcd_of_m_and_q_equals_1} and thus $q$ is invertible modulo $m$, and so $U_{\pi_U - 1} \equiv - q^{-1} \pmod{m}$. By Identity~\eqref{eq:Lucas_identity_V_r_in_terms_of_U_r_and_U_r_minus_1} of Lemma~\ref{lem:U_n_Lucas_sequence_identities_part_1}, the latter congruence, and the fact that $U_{\pi_U} \equiv 0 \pmod{m}$, we have
$$ V_{\pi_U} =  -2q \cdot U_{\pi_U - 1} + p \cdot U_{\pi_U} \equiv -2q \cdot ( - q^{-1} ) \equiv 2 \pmodd{m}.$$
Again, employing Identity~\eqref{eq:Lucas_identity_V_r_in_terms_of_U_r_and_U_r_minus_1} and the residue classes of $U_{\pi_U}$ and $U_{\pi_U + 1}$, we have
$$ V_{\pi_U + 1} =  -2q \cdot U_{\pi_U} + p \cdot U_{\pi_U + 1} \equiv p \cdot 1 \equiv p \pmodd{m},$$
as desired. We conclude that since $V_{\pi_U} \equiv V_0 \pmod{m}$ and $V_{\pi_U + 1} \equiv V_1 \pmod{m}$, the period $\pi_V$ of $\LucasVSeq$ divides the period $\pi_U$ of $\LucasUSeq$.
\end{proof}

Wall gave a condition for when $\pi_F(m) = \pi_S(m)$, where $\GenericSeq$ is some sequence bearing the Fibonacci recurrence but with arbitrary initial terms $S_0$ and $S_1$~\cite[Corollary to Theorem~8]{Wall1960}. We generalize Wall's proof to all Lucas sequences $\LucasFormalUSeq$ and $\LucasFormalVSeq$.

\begin{theorem}\label{thm:equality_of_periods_of_U_n_and_V_n}
Let $m \geq 1$ be given and set $U_n:=\LucasFormalUTerm$ and $V_n:=\LucasFormalVTerm$. Then $\pi_U(m) = \pi_V(m)$ whenever $\gcd(\Delta,m) = 1$, where $\Delta = p^2 - 4q$ is the discriminant of the characteristic polynomial for the Lucas sequences.
\end{theorem}

\begin{proof}
Let $m \geq 1$ be given. We will show that $\gcd(\Delta,m) = 1$ implies $U_{\pi_V - 1} \equiv U_{-1} \pmod{m}$ and $U_{\pi_V} \equiv U_0 \pmod{m}$, and hence $\pi_U$ divides $\pi_V$ and the result follows from Theorem~\ref{thm:period_of_V_n_divides_period_of_U_n}. Since $\LucasVSeq$ has period $\pi_V$, we have $2 = V_0 \equiv V_{\pi_V} \pmod{m}$ and $p = V_1 \equiv V_{\pi_V + 1} \pmod{m}$, leading to the following system of congruences via Equation~\eqref{eq:Lucas_identity_V_r_in_terms_of_U_r_and_U_r_minus_1} of Lemma~\ref{lem:U_n_Lucas_sequence_identities_part_1}:
\begin{equation*}
\left\{ \begin{array}{rll}
V_{\pi_V} - V_0 &= \left( p \cdot U_{\pi_V} - 2q \cdot U_{\pi_V - 1} \right) - 2 &\equiv 0 \pmodd{m} \\
V_{\pi_V + 1} - V_1 &= \left( p \cdot U_{\pi_V + 1} - 2q \cdot U_{\pi_V} \right) - p &\equiv 0 \pmodd{m}.
\end{array} \right.
\end{equation*}
This simplifies to the system of congruences
\begin{equation}\label{eq:system_of_congruences}
\left\{ \begin{array}{rlrl} 
  -2q \cdot U_{\pi_V - 1} &+ &p \cdot U_{\pi_V} &\equiv 2 \pmodd{m}\\
  -pq \cdot U_{\pi_V - 1} &+ &(p^2 - 2q) \cdot U_{\pi_V} &\equiv p \pmodd{m},
\end{array} \right.
\end{equation}
where the second congruence holds by the sequence of equalities
\begin{align*}
p \cdot U_{\pi_V + 1} - 2q \cdot U_{\pi_V}
    &= p \left( p \cdot U_{\pi_V} - q \cdot U_{\pi_V - 1} \right) - 2q \cdot U_{\pi_V}\\
    &= (p^2 - 2q) \cdot U_{\pi_V} - pq \cdot U_{\pi_V - 1}
\end{align*}
with the first equality holding by the $\LucasUSeq$ recurrence. We can solve System~\eqref{eq:system_of_congruences} for the values $U_{\pi_V - 1}$ and $U_{\pi_V}$. Denoting the coefficient matrix by $M$, the system in matrix form is
\[
M \cdot
\begin{pmatrix}
U_{\pi_V - 1}\\
U_{\pi_V}
\end{pmatrix}
=
\begin{pmatrix}
-2q & p\\
-pq & p^2 - 2q
\end{pmatrix}
\begin{pmatrix}
U_{\pi_V - 1}\\
U_{\pi_V}
\end{pmatrix}
\equiv
\begin{pmatrix}
2\\
p
\end{pmatrix} \pmodd{m}.
\]
It is readily verified that $M$ has determinant $D = -q (p^2 - 4q) = -q \Delta$. Whenever we have $\gcd(q\Delta,m) = 1$, it is readily verified that the unique solution to System~\eqref{eq:system_of_congruences} is as follows:
\[
\begin{pmatrix}
U_{\pi_V - 1}\\
U_{\pi_V}
\end{pmatrix}
\equiv
M^{-1} \cdot
\begin{pmatrix}
2\\
p
\end{pmatrix}
=
-\frac{1}{q \Delta} \cdot
\begin{pmatrix}
p^2 - 2q & -p\\
pq & -2q
\end{pmatrix}
\begin{pmatrix}
2\\
p
\end{pmatrix}
\equiv
\begin{pmatrix}
-q^{-1}\\
0
\end{pmatrix} \pmodd{m}.
\]
Thus, $U_{\pi_V - 1} \equiv -q^{-1} \pmod{m}$ and $U_{\pi_V} \equiv 0 \pmod{m}$. But $U_{\pi_U - 1} \equiv - q^{-1} \pmod{m}$ (discussed in the proof of Theorem~\ref{thm:period_of_V_n_divides_period_of_U_n}), and so $U_{-1} \equiv -q^{-1} \pmod{m}$. We conclude that since $U_{\pi_V - 1} \equiv U_{-1} \pmod{m}$ and $U_{\pi_V} \equiv U_0 \pmod{m}$, the period $\pi_U$ of $\LucasFormalUSeq$ divides the period $\pi_V$ of $\LucasFormalVSeq$. However, Theorem~\ref{thm:period_of_V_n_divides_period_of_U_n} states $\pi_V$ divides $\pi_U$, and hence whenever $\gcd(q\Delta,m) = 1$, we have $\pi_U = \pi_V$. In particular, since $\gcd(q,m) = 1$ always holds by Convention~\ref{conv:gcd_of_m_and_q_equals_1}, then the condition $\gcd(q\Delta,m) = 1$ simplifies to $\gcd(\Delta,m) = 1$.
\end{proof}

We conclude this subsection with two more sufficiency criteria for the equality of the periods $\pi_U(m)$ and $\pi_V(m)$. This is attributed to correspondences between second author Mbirika and Christian Ballot. This result below or its corollary are used in the proofs of Theorems~\ref{thm:Jurgen_conjecture_on_existence_of_entry_point_generalized_to_Un_Vn_setting}, \ref{thm:omega_V_equals_2}, \ref{thm:V_n_order_4_vertical_slice_result}, and Corollary~\ref{cor:omega_V_equals_twice_omega_U}.

\begin{theorem}\label{thm:Christian_Ballot_equality_of_periods}
Let $m \geq 1$ be given and set $U_n:=\LucasFormalUTerm$ and $V_n:=\LucasFormalVTerm$. If $e_V(m)$ exists, then the following congruence holds for all $n \in \mathbb{Z}$: 
\begin{align}
    V_{e_V(m) + 1} U_n \equiv V_{e_V(m) + n} \pmodd{m}. \label{eq:Ballot_inspired_congruence}
\end{align}
Moreover, if $\gcd(V_{e_V(m) + 1}, m) = 1$, then $\pi_U(m) = \pi_V(m)$.
\end{theorem}

\begin{proof}
Let $m \geq 1$ be given. We will show by induction that $V_{e_V + 1} U_n \equiv V_{e_V + n} \pmod{m}$ for all $n \geq 0$. Clearly this is true for $n=0,1$, and so the base cases hold. Now suppose that $V_{e_V + 1} U_k \equiv V_{e_V + k} \pmod{m}$ and $V_{e_V + 1} U_{k+1} \equiv V_{e_V + (k+1)} \pmod{m}$ for some $k \geq 0$. Observe that
\begin{align*}
    V_{e_V + (k+2)} &= p V_{e_V + (k+1)} - q V_{e_V + k}\\
    &\equiv p(V_{e_V + 1} U_{k+1}) - q(V_{e_V + 1}U_k) \pmodd{m}\\
    &\equiv V_{e_V + 1} \cdot (p U_{k+1} - q U_k) \pmodd{m}\\
    &\equiv V_{e_V + 1} U_{k+2} \pmodd{m},
\end{align*}
where the first congruence holds by the induction hypothesis. This completes the inductive step, and we conclude that $V_{e_V + 1} U_n \equiv V_{e_V + n} \pmod{m}$ for all $n \geq 0$. Now suppose that $n < 0$. Then $U_n = -q^n U_{-n}$ by Identity~\eqref{eq:Lucas_identity_negative_index} of Lemma~\ref{lem:U_n_Lucas_sequence_identities_part_1}. Since $-n > 0$, we have
\begin{align*}
V_{e_V + 1} U_n = -q^n V_{e_V + 1} U_{-n} &\equiv -q^n V_{e_V + (-n)} \pmodd{m}\\
    &\equiv -q^n \cdot \left(V_{e_V} V_{-n} - q^{-n} V_{e_V - (-n)}\right) \pmodd{m}\\
    &\equiv V_{e_V + n} \pmodd{m},
\end{align*}
where the second congruence holds by Identity~\eqref{eq:Lucas_identity_for_Ballot_congruence} of Lemma~\ref{lem:U_n_Lucas_sequence_identities_part_1}, and the third congruence holds since $V_{e_V} \equiv 0 \pmod{m}$. We conclude that $V_{e_V + 1} U_n \equiv V_{e_V + n} \pmod{m}$ for all $n \in \mathbb{Z}$.

Now assume that $\gcd(V_{e_V + 1}, m) = 1$ holds. We will show that $\pi_U = \pi_V$ follows. Consider the map $f$ from the period terms $\left(U_n \pmod{m}\right)_{n=1}^{\pi_U}$ to $\left(V_n \pmod{m}\right)_{n=1}^{\pi_V}$ defined by $$f\left(U_i \pmodd{m}\right) = V_{e_V + 1} U_i \pmodd{m} \text{ for all } 1 \leq i \leq \pi_U.$$
By Congruence~\eqref{eq:Ballot_inspired_congruence}, this map is well defined. Since $\pi_U = k \pi_V$ for some $k \geq 1$ by Theorem~\ref{thm:period_of_V_n_divides_period_of_U_n}, it suffices to show that $k=1$. Suppose by way of contradiction that $k>1$ and hence $\pi_V < \pi_U$. Then, in particular, we have $f(U_1) = f(U_{\pi_V + 1})$ and $f(U_2) = f(U_{\pi_V + 2})$, and
\begin{align*}
    V_{e_V + 1} U_1 \equiv V_{e_V + 1} U_{\pi_V + 1} \pmodd{m} &\;\Longrightarrow\; U_1 \equiv U_{\pi_V + 1} \pmodd{m}\\
    V_{e_V + 1} U_2 \equiv V_{e_V + 1} U_{\pi_V + 2} \pmodd{m} &\;\Longrightarrow\; U_2 \equiv U_{\pi_V + 2} \pmodd{m},
\end{align*}
which hold since $\gcd(V_{e_V + 1}, m) = 1$. It follows that $(U_1, U_2) \equiv (U_{\pi_V + 1}, U_{\pi_V + 2}) \pmod{m}$ and hence $\pi_U \leq \pi_V$, contradicting the contrary assumption that $\pi_V < \pi_U$. Therefore $k = 1$, and we conclude $\gcd(V_{e_V + 1}, m) = 1$ implies $\pi_U = \pi_V$, as desired.
\end{proof}

\begin{corollary}\label{cor:Sufficiency_condition_for_equality_of_periods}
Let $m > 2$ be given and set $V_n:=\LucasFormalVTerm$. If $e_V(m)$ exists, then we have $\pi_U(m) = \pi_V(m)$ if either condition holds: (1) $p$ is odd, or (2) $p$ is even and $m$ is odd.
\end{corollary}

\begin{proof}
Let $m>2$ be given, and suppose that $p$ is odd. Then $\gcd(V_{e_V}, V_{e_V + 1}) = 1$ by Corollary~\ref{cor:GCD_of_consecutive_terms_in_Vn_is_1_or_2}. Since $m$ divides $V_{e_V}$ and $\gcd(V_{e_V}, V_{e_V + 1}) = 1$, then $\gcd(V_{e_V + 1}, m) = 1$ and so $\pi_U = \pi_V$ by Theorem~\ref{thm:Christian_Ballot_equality_of_periods}.

Now suppose that $p$ is even. Since $m$ divides $V_{e_V}$ and $\gcd(V_{e_V}, V_{e_V + 1}) = 2$ by Corollary~\ref{cor:GCD_of_consecutive_terms_in_Vn_is_1_or_2}, we have $\gcd(V_{e_V + 1}, m) \leq \gcd(V_{e_V + 1}, V_{e_V}) = 2$, and hence $\gcd(V_{e_V + 1}, m) \in \{1,2\}$ . Noting that $V_{e_V + 1}$ is even, it follows that
$$\gcd(V_{e_V + 1}, m) =
\begin{cases}
    1, \text{ if $m$ is odd};\\
    2, \text{ if $m$ is even}.
\end{cases}
$$
Thus if $p$ is even and $m$ is odd, then $\gcd(V_{e_V + 1}, m) = 1$ and so $\pi_U = \pi_V$ by Theorem~\ref{thm:Christian_Ballot_equality_of_periods}.
\end{proof}


\subsection{Entry points and order: the statistics \texorpdfstring{$e_U(m)$}{entry point of m} and \texorpdfstring{$\omega_U(m)$}{order of m}}\label{subsec:entry_points_and_order}

In the Fibonacci setting, it is well known that the value $\omega_F(m)$ equals 1, 2, or 4. This result was proven independently by both Vinson and Robinson in the very first year of the \textit{Fibonacci Quarterly} in 1963~\cite{Vinson1963,Robinson1963}. More recently, in 2013, Renault generalized this result to the Lucas sequences $\LucasFormalUSeq$~\cite{Renault2013}.\footnote{Renault denotes the Lucas sequence $\LucasFormalUSeq$ with characteristic polynomial $x^2 - px + q$ by the term ``$(a,b)$-Fibonacci sequence'' with characteristic polynomial $x^2 - ax - b$. Hence the two sequences are equivalent if we replace the parameters $p$ and $q$ with $a$ and $-b$, respectively.} One main result is the following.
\begin{proposition}[{\cite[Theorem~4(a)]{Renault2013}}]\label{prop:Renault_order_is_1_2_or_4}
Let $m \geq 1$ be given. Then $\omega_U(m)$ divides $2 \cdot \ord_m(q)$, where $\ord_m(q)$ denotes the multiplicative order of $q$ modulo $m$.
\end{proposition}

Given the latter proposition in the setting of the Lucas sequences $\LucasFormalUSeq$, we have the following corollary when $q = \pm 1$.

\begin{corollary}\label{cor:Order_in_Un_when_q_equals_1_or_minus_1}
Let $m\geq 1$ be given and set $U_n:=\LucasFormalUTerm$. If $q = -1$ (respectively, $q=1$), then $\omega_U(m) \in \{1,2,4\}$ (respectively, $\omega_U(m) \in \{1,2\}$). In particular, we have the following:
\begin{align*}
    (\text{Fibonacci sequence}) &\;\;\; \omega_F(m) \in \{1,2,4\}\\
    (\text{Pell sequence}) &\;\;\; \omega_P(m) \in \{1,2,4\}\\
    (\text{balancing sequence}) &\;\;\; \omega_B(m) \in \{1,2\}.
\end{align*}
\end{corollary}

\begin{proof}
Observe that $\ord_m(-1)=1$ if $m=1,2$ and $\ord_m(-1) = 2$ whenever $m > 2$. Hence $\omega_U(m) \in \{1,2,4\}$ when $q=-1$ by Proposition~\ref{prop:Renault_order_is_1_2_or_4}. Furthermore, $\ord_m(1)=1$ for all $m \geq 1$. Hence $\omega_U(m) \in \{1,2\}$ when $q=1$ by Proposition~\ref{prop:Renault_order_is_1_2_or_4}. 
\end{proof}

In the Fibonacci setting, necessary and sufficient conditions on whether the order $\omega_F(m)$ is 1, 2, or 4 is given by Vinson in 1963~\cite[Theorems~1 and 3]{Vinson1963} and more recently by Desmond in 1978~\cite[Theorem~1]{Desmond1978} in the following equivalent form:
        \begin{numcases}{\omega_F(m) = }
        4, &\text{ if and only if $m > 2$ and $e_F(m)$ is odd};\nonumber\\
	1, &\text{ if and only if $m = 1 \text{ or } 2$ or $\frac{\pi_F(m)}{2}$ is odd};\nonumber\\
        2, &\text{ if and only if $e_F(m)$ and $\frac{\pi_F(m)}{2}$ are both even}.\nonumber
	\end{numcases}
Since certain applications in our paper involve graphical behavior of the Lucas sequences dependent on the order $\omega_U(m)$, we generalize Vinson/Desmond's results above to $\LucasFormalUSeq$ in the case when $q = -1$. In particular, for arbitrary values $p$ and $m$, we have
\begin{align}
m > 2 \text{ and } e_U(m) \text{ is odd} &\;\Longleftrightarrow\; \omega_U(m) = 4;\label{eq:New_Vinson_1}\tag{Theorem~\ref{thm:order_equals_4_for_certain_Lucas_sequences}}\\
m = 1 \text{ or } 2 \text{ or } \frac{\pi_U(m)}{2} \text{ is odd} &\;\Longrightarrow\; \omega_U(m) = 1;\label{eq:New_Vinson_2}\tag{\eqref{eq:order_1_left_to_right} of Theorem~\ref{thm:order_equals_1_for_certain_Lucas_sequences}}\\
 e_U(m) \text{ and } \frac{\pi_U(m)}{2} \text{ are both even} &\;\Longleftarrow\; \omega_U(m) = 2;\label{eq:New_Vinson_3}\tag{\eqref{eq:order_2_right_to_left} of Theorem~\ref{thm:order_equals_2_for_certain_Lucas_sequences}}
\end{align}
and if either $p$ or $m$ is odd, then the converses \eqref{eq:order_1_right_to_left} and \eqref{eq:order_2_left_to_right}, respectively, of \eqref{eq:order_1_left_to_right} and \eqref{eq:order_2_right_to_left} in the two implications above also hold.

\begin{theorem}\label{thm:order_equals_4_for_certain_Lucas_sequences}
Set $U_n:=\LucasFormalUTerm$ and $V_n:=\LucasFormalVTerm$ with $q = -1$. Then for $m > 2$, the following biconditional holds:
$$ e_U(m) \text{ is odd}\;\Longleftrightarrow\; \omega_U(m) = 4.$$
\end{theorem}

\begin{proof}
Let $m>2$ be given, and assume that $e_U$ is odd. Since $q = -1$, we have $\omega_U \in \{1,2,4\}$ by Corollary~\ref{cor:Order_in_Un_when_q_equals_1_or_minus_1}, and thus  $\pi_U = k e_U$ for some $k \in \{1, 2, 4\}$. Furthermore by Corollary~\ref{cor:Period_in_Un_is_even_when_q_equals_minus_1}, we know that $\pi_U$ is even for $m>2$. Since $\pi_U$ is even and $e_U$ is odd, $k$ cannot equal 1. Hence $k$ is either 2 or 4; that is, $\pi_U$ equals $2e_U$ or $4e_U$. Observe that $U_{2e_U} \equiv 0 \pmod{m}$ by Theorem~\ref{thm:equally_spaced_zeros_in_U_n_sequences}. We will show that $U_{2e_U + 1} \not\equiv 1 \pmod{m}$, and hence $\pi_U$ cannot equal $2e_U$, which forces $k = 4$. Observe that
$$U_{2e_U + 1} = (-1)^{e_U} + U_{e_U} V_{e_U + 1} \equiv -1 \pmodd{m},$$
where the equality holds by Identity~\eqref{eq:Lucas_identity_odd_subscript_1} of Lemma~\ref{lem:odd_subscript_Un_result}, and the congruence holds since $e_U$ is odd and $U_{e_U} \equiv 0 \pmod{m}$. Thus $\pi_U \neq 2e_U$ and $k=4$ is forced. That is, $\pi_U = 4e_U$ and hence $\omega_U = 4$, as desired. This proves the sufficiency condition.

Now assume that $\omega_U = 4$, and hence $\pi_U = 4e_U$. Suppose by way of contradiction that $e_U$ is even. As in the proof for the sufficiency condition, we have $U_{2e_U + 1} = (-1)^{e_U} + U_{e_U} V_{e_U + 1}$. It follows that $U_{2e_U + 1} \equiv 1 \pmod{m}$ since $e_U$ is even and $U_{e_U} \equiv 0 \pmod{m}$. Thus we have $(U_{2 e_U}, U_{2e_U + 1}) \equiv (0,1) \pmod{m}$, but this contradicts the fact that $\pi_U = 4e_U$. Thus $e_U$ is odd. This proves the necessity condition.
\end{proof}

The following corollary is immediate from the proof of the preceding theorem. This corollary is used in Theorem~\ref{thm:antipodal_result_for_U_n_sequence_with_order_4} in the applications section.

\begin{corollary}\label{cor:left_of_butt_in_Un_is_minus_one}
Set $U_n:=\LucasFormalUTerm$ with $q = -1$. For $m > 2$, if $\omega_U(m) = 4$, it follows that $U_{2e_U(m) + 1} \equiv -1 \pmod{m}$.
\end{corollary}

\begin{theorem}\label{thm:order_equals_1_for_certain_Lucas_sequences}
Set $U_n:=\LucasFormalUTerm$ with $q = -1$. Then for $m \geq 1$, we have
\begin{align}
m = 1 \text{ or } 2 \text{ or } \frac{\pi_U(m)}{2} \text{ is odd} &\;\Longrightarrow\; \omega_U(m) = 1 &\blue{\text{$[$for arbitrary $p$ and $m]$}},\label{eq:order_1_left_to_right}\\
m = 1 \text{ or } 2 \text{ or } \frac{\pi_U(m)}{2} \text{ is odd} &\;\Longleftarrow\; \omega_U(m) = 1 &\blue{\text{$[$for $p$ or $m$ odd$]$}}.\label{eq:order_1_right_to_left}
\end{align}
\end{theorem}

\begin{proof}
Let $m \geq 1$ be given. If $m = 1 \text{ or } 2$, then by Remark~\ref{rem:when_m_equals_1_or_2_regarding_statistics}, we have $\omega_U = 1$, and we are done. So assume that $m > 2$ and $\frac{\pi_U}{2}$ is odd. Since $q = -1$, we have $\omega_U \in \{1,2,4\}$ by Corollary~\ref{cor:Order_in_Un_when_q_equals_1_or_minus_1}, and thus  $\pi_U = k e_U$ for some $k \in \{1, 2, 4\}$. If $k = 4$, then $\frac{\pi_U}{2} = 2 e_U$ is even, contradicting the assumption that $\frac{\pi_U}{2}$ is odd, so $k \neq 4$. If $k=2$, then $\frac{\pi_U}{2} = e_U$, and so $e_U$ is odd. However, that would force $\omega_U = 4$ by Theorem~\ref{thm:order_equals_4_for_certain_Lucas_sequences} and so $k = 4$, again a contradiction. We conclude that $k=1$. That is, $\pi_U = e_U$ and hence $\omega_U = 1$, as desired. This proves Implication~\eqref{eq:order_1_left_to_right}.

Now assume that $\omega_U = 1$ and $m>2$, and suppose that $p$ or $m$ is odd. Since $\omega_U \neq 4$, then $e_U$ is even by Theorem~\ref{thm:order_equals_4_for_certain_Lucas_sequences}, and hence $\frac{e_U}{2}$ is an integer. Suppose by way of contradiction that $\frac{\pi_U}{2}$ is even. Then since $\omega_U = 1$, we have $\pi_U = e_U$ and so $\frac{e_U}{2}$ is also even. Setting $n := \frac{e_U}{2}$ in Lemma~\ref{lem:Desmond_lemma_1_generalization}, we have $U_{2\left(\frac{e_U}{2}\right)} \equiv 0 \pmod{m}$ and $U_{2\left(\frac{e_U}{2}\right)+1} \equiv (-1)^{\frac{e_U}{2}} \pmod{m}$, where the second congruence holds since $e_U = \pi_U$ and $\frac{e_U}{2}$ even implies
$$ U_{2\left(\frac{e_U}{2}\right)+1} = U_{e_U + 1} = U_{\pi_U + 1} \equiv 1 \equiv (-1)^{\frac{e_U}{2}} \pmodd{m}.$$
Thus $U_{\frac{e_U}{2}} \equiv 0 \pmod{m}$ is forced by Implication~\eqref{eq:Desmond_lemma_1_left_to_right} of Lemma~\ref{lem:Desmond_lemma_1_generalization} since $p$ or $m$ is odd, contradicting the fact that $e_U$ is the entry point. Thus $\frac{\pi_U}{2}$ is odd, as desired. This proves Implication~\eqref{eq:order_1_right_to_left}.
\end{proof}

\begin{theorem}\label{thm:order_equals_2_for_certain_Lucas_sequences}
Set $U_n:=\LucasFormalUTerm$ with $q = -1$. Then for $m > 2$, we have
\begin{align}
e_U(m) \text{ and } \frac{\pi_U(m)}{2} \text{ are both even} &\;\Longleftarrow\; \omega_U(m) = 2 &\blue{\text{$[$for arbitrary $p$ and $m]$}},\label{eq:order_2_right_to_left}\\
e_U(m) \text{ and } \frac{\pi_U(m)}{2} \text{ are both even} &\;\;\Longrightarrow\; \omega_U(m) = 2 &\blue{\text{$[$for $p$ or $m$ odd$]$}}.\label{eq:order_2_left_to_right}
\end{align}
\end{theorem}

\begin{proof}
Let $m > 2$ be given, and assume that $\omega_U = 2$. Since $\omega_U \neq 4$, then $e_U$ is even by Theorem~\ref{thm:order_equals_4_for_certain_Lucas_sequences}. Moreover $\omega_U = 2$ implies that $\frac{\pi_U}{e_U} = 2$ by definition and so $e_U = \frac{\pi_U}{2}$, and hence $\frac{\pi_U}{2}$ is even, as desired. This proves Implication~\eqref{eq:order_2_right_to_left}.

Now assume that $e_U$ and $\frac{\pi_U}{2}$ are both even, and suppose that $p$ or $m$ is odd. Since $q = -1$, we have $\omega_U \in \{1,2,4\}$ by Corollary~\ref{cor:Order_in_Un_when_q_equals_1_or_minus_1}, but $e_U$ is even and so $\omega_U \neq 4$ by Theorem~\ref{thm:order_equals_4_for_certain_Lucas_sequences}. Moreover, since $\frac{\pi_U}{2}$ is even, then in particular $m \neq 1,2$ by Remark~\ref{rem:when_m_equals_1_or_2_regarding_statistics}. Therefore, by the contrapositive of Implication~\eqref{eq:order_1_right_to_left} of Theorem~\ref{thm:order_equals_1_for_certain_Lucas_sequences} and the fact that $p$ or $m$ is odd, we know $\omega_U \neq 1$ and hence $\omega_U = 2$, as desired. This proves Implication~\eqref{eq:order_2_left_to_right}.
\end{proof}


\subsection{Some consequences of the existence of the value \texorpdfstring{$e_V(m)$}{entry point of m in Vn}}\label{subsec:consequences_of_existence_of_entry_point_in_V_n}

Although the value $e_U(m)$ exists for all $m \geq 1$ in the Lucas sequences $\LucasFormalUSeq$ whenever $\gcd(q,m) = 1$ (see Convention~\ref{conv:gcd_of_m_and_q_equals_1}), it is known that for many $m$ values with $\gcd(q,m) = 1$, the value $e_V(m)$ fails to exist in the Lucas sequences $\LucasFormalVSeq$. In this subsection, we offer some surprising consequences when the value $e_V(m)$ does exist.

\boitegrise{
\begin{convention}[restriction of $p$ and moduli $m$ values]\label{conv:restriction_of_p_and_m_for_equality_of_periods}
\vspace{-.15in}
For the remainder of this section, we are limiting our parameters to either of the following two conditions:
\begin{center}
(1) $p$ is odd and $m>2$ is arbitrary, or (2) $p$ is even and $m>2$ is odd.
\end{center}
We establish these restrictions since each is sufficient for the existence of $e_V(m)$ to imply $\pi_U(m) = \pi_V(m)$, as given in Corollary~\ref{cor:Sufficiency_condition_for_equality_of_periods}. We note that this new restriction of parameters $p$ and $m$ does not affect the validity of Theorems~\ref{thm:order_equals_1_for_certain_Lucas_sequences} and \ref{thm:order_equals_2_for_certain_Lucas_sequences}  in Subsection~\ref{subsec:entry_points_and_order}, which are the only theorems in that subsection that have $p$ and $m$ restrictions.
\vspace{-.3in}
\end{convention}}{0.95\textwidth}

\begin{theorem}\label{thm:Jurgen_conjecture_on_existence_of_entry_point_generalized_to_Un_Vn_setting}
Set $U_n:=\LucasFormalUTerm$ and $V_n:=\LucasFormalVTerm$ with $q = \pm 1$. For $m>2$, if $e_V(m)$ exists then $2 e_V(m) = e_U(m)$.
\end{theorem}

\begin{proof}
Let $m>2$ be given. By Identity~\eqref{eq:Lucas_identity_U_2n_equals_U_n_V_n} of Lemma~\ref{lem:U_n_Lucas_sequence_identities_part_1}, we have $U_{2 e_V} = U_{e_V} V_{e_V}$, and it follows that
$$ V_{e_V} \equiv 0 \pmodd{m} \implies U_{2 e_V} \equiv 0 \pmodd{m} \implies e_U \text{ divides } 2 e_V,$$
where the second implication follows from Theorem~\ref{thm:equally_spaced_zeros_in_U_n_sequences}. We claim that $e_U$ is even for otherwise $e_U$ is odd and since $e_U$ divides $2 e_V$, that forces $e_U$ to divide $e_V$. Hence $U_{e_U}$ divides $U_{e_V}$ by Identity~\eqref{eq:Lucas_identity_U_r_divides_U_s} of Lemma~\ref{lem:U_n_Lucas_sequence_identities_part_1}, and thus $m$ would divide both $U_{e_V}$ and $V_{e_V}$, so $m$ divides $\gcd(U_{e_V}, V_{e_V}) \leq 2$, where the inequality holds by Identity~\eqref{eq:Lucas_identity_GCD_of_Un_and_Vn} of Lemma~\ref{lem:U_n_Lucas_sequence_identities_part_1}. Thus $m \leq 2$, a contradiction, so $e_U$ is even. If $q = -1$, then $e_U$ being even implies $\omega_U \neq 4$ by Theorem~\ref{thm:order_equals_4_for_certain_Lucas_sequences}, and hence $\omega_U$ is either 1 or 2 by Corollary~\ref{cor:Order_in_Un_when_q_equals_1_or_minus_1}. Furthermore, if $q = 1$, then again $\omega_U$ is either 1 or 2 by Corollary~\ref{cor:Order_in_Un_when_q_equals_1_or_minus_1}. So if $q = \pm 1$, we have two cases to consider.

\medskip

\noindent \textbf{(CASE 1):} Assume that $\omega_U = 1$. Since $U_{2 e_V} \equiv 0 \pmod{m}$, then $2 e_V = k e_U$ for some $k \geq 1$ by Theorem~\ref{thm:equally_spaced_zeros_in_U_n_sequences}. It suffices to show $k=1$. Suppose by way of contradiction that $k>1$. Then $2 e_V \geq 2 e_U$ and hence $e_V \geq e_U$. Since $e_V$ exists, then $\pi_V = \pi_U$ by Corollary~\ref{cor:Sufficiency_condition_for_equality_of_periods}. Moreover $V_{\pi_V} \equiv 2 \not\equiv 0 \pmod{m}$ and hence $e_V < \pi_V$. Since $\omega_U = 1$, we have $\pi_U = e_U$ and thus $e_V < \pi_V = \pi_U = e_U$, contradicting $e_V \geq e_U$. Hence $k=1$ and $2 e_V = e_U$ holds, as desired.

\medskip

\noindent \textbf{(CASE 2):} Assume that $\omega_U = 2$. Then the only index $r$ such that $U_r \equiv 0 \pmod{m}$ with $0 < r < \pi_U$ occurs when we set $r := e_U$. Since $U_{2 e_V} \equiv 0 \pmod{m}$, it suffices to show that $0 < 2e_V < \pi_U$ and hence $2 e_V = e_U$ is forced, and the result follows. To this end, suppose by way of contradiction that $2e_V \geq \pi_U$. Then since $\omega_U = 2$ implies that $e_U = \frac{\pi_U}{2}$, we have
$$ 2 e_V \geq \pi_U \;\implies\; e_V \geq \frac{\pi_U}{2} \;\implies\; e_V \geq e_U.$$
Since $e_V$ exists, then $\pi_V = \pi_U$ by Corollary~\ref{cor:Sufficiency_condition_for_equality_of_periods}. Moreover $V_{\pi_V} \equiv 2 \not\equiv 0 \pmod{m}$ and hence $e_V < \pi_V$. It follows that $e_U \leq e_V <  \pi_V = \pi_U$. Setting $n := e_U$ in Identity~\eqref{eq:Ballot_inspired_congruence} of Theorem~\ref{thm:Christian_Ballot_equality_of_periods}, we have $V_{e_V + 1} U_{e_U} \equiv V_{e_V + e_U} \pmod{m}$ and so $V_{e_V + e_U} \equiv 0 \pmod{m}$ follows since $U_{e_U} \equiv 0 \pmod{m}$. We have $e_U \leq e_V < \pi_U$, and we claim $e_U \neq e_V$. Otherwise if $e_U = e_V$, then $2e_V = 2e_U = \pi_U = \pi_V$, where the second equality holds since $\omega_U = 2$, and so
$$ 2 \equiv V_{\pi_V} = V_{2e_V} = V_{e_V + e_V}  = V_{e_V + e_U} \equiv 0 \pmodd{m},$$
contradicting $m>2$. Thus $e_U \neq e_V$ and $e_U < e_V < \pi_V$, and so the first zero in $\LucasVSeq$ modulo $m$ occurs in the set $\{V_{e_U + 1}, V_{e_U + 2}, \ldots, V_{\pi_V - 1}\}$ modulo $m$. However, observe that
$$ \pi_V = \pi_U = 2 e_U = e_U + e_U < e_V + e_U < \pi_V + e_U, $$
and thus the index $e_V + e_U$ lies in between $\pi_V$ and $\pi_V + e_U$. Hence there exists a zero, namely $V_{e_V + e_U} \pmod{m}$, in the set $\{V_{\pi_V}, V_{\pi_V + 1}, V_{\pi_V + 2}, \ldots, V_{\pi_V + e_U}\}$ modulo $m$ and therefore also in $\{V_0, V_1, V_2, \ldots, V_{e_U}\}$ modulo $m$, contradicting $e_V > e_U$. We conclude that $0 < 2e_V < \pi_U$ and thus $2e_V = e_U$, as desired.
\end{proof}

The following corollary, which arises from the previous theorem, establishes a relationship between $\omega_U(m)$ and $\omega_V(m)$ when $q = \pm 1$.

\begin{corollary}\label{cor:omega_V_equals_twice_omega_U}
Set $U_n:=\LucasFormalUTerm$ and $V_n:=\LucasFormalVTerm$ with $q = \pm 1$. Assume that $e_V$ exists. Then for $m>2$,  we have $\omega_V(m) = 2 \omega_U(m)$.
\end{corollary}

\begin{proof}
Let $m>2$ be given, and assume that $e_V$ exists. Observe the following sequence of equalities:
$$ \omega_U = \frac{\pi_U}{e_U} = \frac{\pi_V}{2 e_V} = \frac{1}{2} \omega_V, $$
where the second equality holds since $\pi_U = \pi_V$ by Corollary~\ref{cor:Sufficiency_condition_for_equality_of_periods} and $e_U = 2 e_V$ by Theorem~\ref{thm:Jurgen_conjecture_on_existence_of_entry_point_generalized_to_Un_Vn_setting}. We conclude that $\omega_V = 2 \omega_U$, as desired.
\end{proof}

\begin{theorem}\label{thm:consequence_of_existence_of_odd_entry_point_for_V_n}
Set $U_n:=\LucasFormalUTerm$ and $V_n:=\LucasFormalVTerm$ with $q = -1$. Then for $m>2$, we have
\begin{align}
e_V(m) \text{ exists and is odd} &\;\Longrightarrow\; \omega_U(m) = 1 \label{eq:order_1_V_entry_consequence_left_to_right}\\
e_V(m) \text{ exists and is odd} &\;\Longleftarrow\; \omega_U(m) = 1 &\blue{\text{$[$for $p$ or $m$ odd$]$}}.\label{eq:order_1_V_entry_consequence_right_to_left}
\end{align}
\end{theorem}

\begin{proof}
Let $m>2$ be given, and assume that $e_V$ exists and is odd. Since $e_V$ exists, we have $2 e_V = e_U$ by Theorem~\ref{thm:Jurgen_conjecture_on_existence_of_entry_point_generalized_to_Un_Vn_setting}, and hence $\frac{e_U}{2}$ is odd since $e_V$ is odd by assumption. Setting $n := \frac{e_U}{2}$ in Identity~\eqref{eq:Lucas_identity_odd_subscript_2} of Lemma~\ref{lem:odd_subscript_Un_result}, we have the second equality in the sequence of equalities
$$ U_{e_U + 1} = U_{2\left(\frac{e_U}{2}\right)+1} = -(-1)^{\frac{e_U}{2}} + U_{\frac{e_U}{2} + 1} V_{\frac{e_U}{2}} = (-1)^{e_V + 1} + U_{e_V + 1} V_{e_V}, $$
where the third equality follows since $e_V = \frac{e_U}{2}$. Therefore $U_{e_U + 1} \equiv 1 \pmod{m}$ since $e_V$ is odd and $V_{e_V} \equiv 0 \pmod{m}$. Thus we have $(U_{e_U}, U_{e_U + 1}) \equiv (0,1) \pmod{m}$. Therefore $\pi_U = e_U$ and hence $\omega_U = 1$, as desired. This proves Implication~\eqref{eq:order_1_V_entry_consequence_left_to_right}.

Now assume that $\omega_U = 1$ and suppose that $p$ or $m$ is odd. Then since $m>2$ and $\omega_U = 1$, it follows that $\frac{\pi_U}{2}$ is odd by Implication~\eqref{eq:order_1_right_to_left} of Theorem~\ref{thm:order_equals_1_for_certain_Lucas_sequences}. Setting $n := \frac{\pi_U}{2}$ in Lemma~\ref{lem:Desmond_lemma_2_generalization}, we have $U_{2\left(\frac{\pi_U}{2}\right)} \equiv 0 \pmod{m}$ and $U_{2\left(\frac{\pi_U}{2}\right)+1} \equiv (-1)^{\frac{\pi_U}{2} + 1} \pmod{m}$, and so $V_{\frac{\pi_U}{2}} \equiv 0 \pmod{m}$. Thus $e_V$ exists. Observe that $\pi_U = e_U = 2 e_V$, where the first equality holds since $\omega_U = 1$, and the second equality holds by Theorem~\ref{thm:Jurgen_conjecture_on_existence_of_entry_point_generalized_to_Un_Vn_setting}. Thus $e_V = \frac{\pi_U}{2}$ is odd, as desired. This proves Implication~\eqref{eq:order_1_V_entry_consequence_right_to_left}.
\end{proof}

\begin{theorem}\label{thm:consequence_of_existence_of_even_entry_point_for_V_n}
Set $U_n:=\LucasFormalUTerm$ and $V_n:=\LucasFormalVTerm$ with $q = -1$. Then for $m>2$, the following biconditional holds for $p$ or $m$ odd:
$$ e_V(m) \text{ exists and is even} \;\Longleftrightarrow\; \omega_U(m) = 2 \text{ and } U_{e_U(m) + 1} \equiv -1 \pmodd{m}.$$
\end{theorem}

\begin{proof}
Let $m>2$ be given, and assume that $e_V$ exists and is even. Then $2 e_V = e_U$ by Theorem~\ref{thm:Jurgen_conjecture_on_existence_of_entry_point_generalized_to_Un_Vn_setting}, and hence $e_U$ is even. Morever since $e_V$ is even, then $e_U \equiv 0 \pmod{4}$. But $e_U$ divides $\pi_U$ by Corollary~\ref{cor:entry_point_divides_period_in_Un_sequence}, and so $\pi_U \equiv 0 \pmod{4}$ also. Hence $\frac{\pi_U}{2}$ is even. Since both $e_U$ and $\frac{\pi_U}{2}$ are even, and $p$ or $m$ is odd, then $\omega_U = 2$ by Implication~\eqref{eq:order_2_left_to_right} of Theorem~\ref{thm:order_equals_2_for_certain_Lucas_sequences}. Finally since $e_U = 2 e_V$, we have the following:
$$ U_{e_U + 1} = U_{2e_V + 1} = -(-1)^{e_V} + U_{e_V + 1} V_{e_V} \equiv -1 \pmodd{m}, $$
where the second equality holds by Identity~\eqref{eq:Lucas_identity_odd_subscript_2} of Lemma~\ref{lem:odd_subscript_Un_result}, and the congruence holds since $e_V$ is even and $V_{e_V} \equiv 0 \pmod{m}$. This proves the sufficiency condition.

Now assume that $\omega_U = 2$ and $U_{e_U + 1} \equiv -1 \pmod{m}$. Since $\omega_U = 2$, then $e_U$ is even by Implication~\eqref{eq:order_2_right_to_left} of Theorem~\ref{thm:order_equals_2_for_certain_Lucas_sequences}. We claim that $\frac{e_U}{2}$ is even. Suppose by way of contradiction that $\frac{e_U}{2}$ is odd. Then $U_{e_U + 1} \equiv -1 \pmod{m}$ implies $U_{e_U + 1} \equiv (-1)^{\frac{e_U}{2}} \pmod{m}$. Setting $n := \frac{e_U}{2}$ in Implication~\eqref{eq:Desmond_lemma_1_left_to_right} of Lemma~\ref{lem:Desmond_lemma_1_generalization}, which we can use since $p$ or $m$ is odd, we have $U_{2\left(\frac{e_U}{2}\right)} \equiv 0 \pmod{m}$ and $U_{2\left(\frac{e_U}{2}\right)+1} \equiv (-1)^{\frac{e_U}{2}} \pmod{m}$, and so $U_{\frac{e_U}{2}} \equiv 0 \pmod{m}$, contradicting that $e_U$ is the entry point. Hence $\frac{e_U}{2}$ is even. Then $U_{e_U + 1} \equiv -1 \pmod{m}$ implies $U_{e_U + 1} \equiv -(-1)^{\frac{e_U}{2}} \pmod{m}$. Therefore, setting $n := \frac{e_U}{2}$ in Lemma~\ref{lem:Desmond_lemma_2_generalization}, we have $U_{2\left(\frac{e_U}{2}\right)} \equiv 0 \pmod{m}$ and $U_{2\left(\frac{e_U}{2}\right)+1} \equiv -(-1)^{\frac{e_U}{2}} \pmod{m}$, and so $V_{\frac{e_U}{2}} \equiv 0 \pmod{m}$. Thus $e_V$ exists. Therefore by Theorem~\ref{thm:Jurgen_conjecture_on_existence_of_entry_point_generalized_to_Un_Vn_setting}, we have $2 e_V = e_U$ and hence $e_V = \frac{e_U}{2}$ is even. This proves the necessity condition.
\end{proof}

\begin{remark}
The necessity condition in Theorem~\ref{thm:consequence_of_existence_of_odd_entry_point_for_V_n} for $e_V(m)$ to be odd has only one condition, namely $\omega_U(m) = 1$. In contrast to that, the necessity condition in Theorem~\ref{thm:consequence_of_existence_of_even_entry_point_for_V_n} for $e_V(m)$ to be even requires that $\omega_U(m)=2$ and additionally that $U_{e_U(m) + 1} \equiv -1 \pmod{m}$ holds. This second condition is vital as one can readily verify in the $\FibSeq$ and $\LucSeq$ modulo 8 setting, wherein we have $\omega_F(8) = 2$ and $e_F(8) = 6$, but $F_7 \equiv 5 \not\equiv -1 \pmod{8}$. It follows that $e_L(8)$ does not exist.
\end{remark}

\begin{theorem}\label{thm:consequence_of_existence_of_entry_point_for_V_n_when_order_in_U_n_is_4}
Set $U_n:=\LucasFormalUTerm$ and $V_n:=\LucasFormalVTerm$ with $q = -1$. Then for $m>2$, the following implication holds:
$$ e_V(m) \text{ exists} \;\Longrightarrow\; \omega_U(m) \neq 4.$$
\end{theorem}

\begin{proof}
Let $m>2$ be given, and assume that $e_V$ exists. Then $e_U = 2 e_V$ by Theorem~\ref{thm:Jurgen_conjecture_on_existence_of_entry_point_generalized_to_Un_Vn_setting}, and so $e_U$ is even. Therefore by Theorem~\ref{thm:order_equals_4_for_certain_Lucas_sequences}, it follows that $\omega_U \neq 4$.
\end{proof}


\subsection{Order: the statistic \texorpdfstring{$\omega_V(m)$}{order of m}}\label{subsec:order_in_V_n_sequence}

If the entry point $e_V(m)$ does not exist, then the order $\omega_V(m)$ similarly cannot exist. However, in this subsection when $e_V(m)$ exists and $q = -1$, we present necessary and sufficient conditions that guarantee a specific value $\omega_V(m) \in \{2, 4\}$.

\begin{theorem}\label{thm:omega_V_equals_2}
Set $U_n:=\LucasFormalUTerm$ and $V_n:=\LucasFormalVTerm$ with $q = -1$. Assume that $e_V(m)$ exists. Then for $m>2$, the following biconditionals holds:
$$ e_U(m) \equiv 2 \pmodd{4} \;\Longleftrightarrow\; e_V(m) \text{ is odd} \;\Longleftrightarrow\; \omega_V(m) = 2.$$
\end{theorem}

\begin{proof}
Let $m>2$ be given, and assume that $e_U \equiv 2 \pmod{4}$. Then $e_U = 4k + 2$ for some $k \geq 0$. Since $2 e_V = e_U$ by Theorem~\ref{thm:Jurgen_conjecture_on_existence_of_entry_point_generalized_to_Un_Vn_setting}, then $2 e_V = 4k + 2$ and thus $e_V = 2k + 1$ is odd. This proves the sufficiency condition of the first biconditional.

Now assume that $e_V$ is odd. Then $\omega_U = 1$ by Implication~\eqref{eq:order_1_V_entry_consequence_left_to_right} of Theorem~\ref{thm:consequence_of_existence_of_odd_entry_point_for_V_n}. Since $\omega_V = 2\omega_U$ by Corollary~\ref{cor:omega_V_equals_twice_omega_U}, we have $\omega_V = 2$. This proves the sufficiency condition of the second biconditional.

Finally, assume that $\omega_V = 2$. We will show that this implies $e_U \equiv 2 \pmod{4}$, thereby proving that the two biconditionals hold.  Since $\omega_V = 2\omega_U$ by Corollary~\ref{cor:omega_V_equals_twice_omega_U}, we have $\omega_U = 1$. Thus $e_V$ is odd by Implication~\eqref{eq:order_1_V_entry_consequence_right_to_left} of Theorem~\ref{thm:consequence_of_existence_of_odd_entry_point_for_V_n}, and so $e_V = 2k + 1$ for some $k \geq 0$. Since $2 e_V = e_U$ by Theorem~\ref{thm:Jurgen_conjecture_on_existence_of_entry_point_generalized_to_Un_Vn_setting}, it follows that $e_U = 2e_V = 4k + 2$. We conclude that $e_U \equiv 2 \pmod{4}$, as desired.
\end{proof}

\begin{theorem}\label{thm:omega_V_equals_4}
Set $U_n:=\LucasFormalUTerm$ and $V_n:=\LucasFormalVTerm$ with $q = -1$. Assume that $e_V(m)$ exists. Then for $m>2$, the following biconditionals holds:
$$ e_U(m) \equiv 0 \pmodd{4} \;\Longleftrightarrow\; e_V(m) \text{ is even} \;\Longleftrightarrow\; \omega_V(m) = 4.$$
\end{theorem}

\begin{proof}
Let $m>2$ be given, and assume $e_U \equiv 0 \pmod{4}$. Then $e_U = 4k$ for some $k \geq 1$. Since $2 e_V = e_U$ by Theorem~\ref{thm:Jurgen_conjecture_on_existence_of_entry_point_generalized_to_Un_Vn_setting}, then $2 e_V = 4k$ and thus $e_V = 2k$ is even. This proves the sufficiency condition of the first biconditional.

Now assume that $e_V$ is even. Then $\omega_U = 2$ by Theorem~\ref{thm:consequence_of_existence_of_even_entry_point_for_V_n}. Since $\omega_V = 2\omega_U$ by Corollary~\ref{cor:omega_V_equals_twice_omega_U}, we have $\omega_V = 4$. This proves the sufficiency condition of the second biconditional.

Finally, assume that $\omega_V = 4$. We will show that this implies $e_U \equiv 0 \pmod{4}$, thereby proving that the two biconditionals hold.  Since $\omega_V = 2\omega_U$ by Corollary~\ref{cor:omega_V_equals_twice_omega_U}, we have $\omega_U = 2$. Thus $e_U$ is even by Implication~\eqref{eq:order_2_right_to_left} of Theorem~\ref{thm:order_equals_2_for_certain_Lucas_sequences}, and hence either $e_U \equiv 0 \pmod{4}$ or $e_U \equiv 2 \pmod{4}$ holds. Suppose by way of contradiction that $e_U \equiv 0 \pmod{4}$ does not hold. Then $e_U \equiv 2 \pmod{4}$ and hence $\omega_V = 2$ by Theorem~\ref{thm:omega_V_equals_2}, contradicting the assumption that $\omega_V = 4$. We conclude that $e_U \equiv 0 \pmod{4}$, as desired.
\end{proof}


\section{Applications: patterns in the Lucas sequences}\label{sec:applications}

In this section, we take a graphical approach to the Lucas sequences. Placing the terms of the fundamental periods in a circle, we uncover surprising patterns which would otherwise be overlooked. To prove our assertions, we utilize theory presented in the previous section.


\subsection{Antipodal sums in \texorpdfstring{$\LucasFormalUSeq$}{Un} and \texorpdfstring{$\LucasFormalVSeq$}{Vn} when \texorpdfstring{$q = -1$}{q is -1} and \texorpdfstring{$\omega(m) = 4$}{order is 4}}

Recall that in the Subsection~\ref{subsec:motivation}, we noted that the antipodal points in the fundamental period of $\FibSeq$ modulo 10 in Figure~\ref{fig:FibTen_circle} are additive inverses of each other. In particular, we have Identity~\eqref{eq:Dan_aBa_Miko_antipodal_result}, proven by Guyer, Mbirika, and Scott~\cite[Theorem~5.3]{Guyer_Mbirika_Scott2024}:
\begin{align*}
    F_{n} +F_{n+\frac{\pi_{F}(10)}{2}} \equiv 0 \pmodd{10}.
\end{align*}
In this subsection, we prove that this result also holds for $\LucasFormalUSeq$ and $\LucasFormalVSeq$ for all $m>2$ when $q = -1$ and $\omega_U(m) = 4$ or $\omega_V(m) = 4$.

\begin{theorem}\label{thm:antipodal_result_for_U_n_sequence_with_order_4}
Let $m > 2$ be given, and set $U_n:=\LucasFormalUTerm$ with $q = -1$. Then for all $n \in \mathbb{Z}$, we have
\begin{align*}
    \omega_U(m) = 4 &\;\implies\; U_{n} + U_{n+\frac{\pi_U(m)}{2}} \equiv 0 \pmodd{m}.
\end{align*}
\end{theorem}

\begin{proof}
Let $m > 2$ be given, and assume that $\omega_U = 4$. Then we have $\frac{\pi_U}{e_U}=4$ and hence $\frac{\pi_U}{2} = 2e_U$. It follows that
    \begin{align*}
        U_n + U_{n + \frac{\pi_U}{2}} &= U_n + U_{n + 2e_U}\\
        &= U_n + \left( U_{n-1} U_{2e_U} + U_{n} U_{2e_U + 1} \right) &\text{by Identity~\eqref{eq:Lucas_identity_Vorobiev_U_n_like_result} of Lemma~\ref{lem:U_n_Lucas_sequence_identities_part_1}}\\
        &\equiv U_n + U_{n} U_{2e_U + 1} \pmodd{m} &\text{since $U_{2e_U} \equiv 0 \pmodd{m}$ by Theorem~\ref{thm:equally_spaced_zeros_in_U_n_sequences}}\\
        &\equiv U_n + U_n \cdot (-1) \pmodd{m} &\text{by Corollary~\ref{cor:left_of_butt_in_Un_is_minus_one}}\\
        &\equiv 0 \pmodd{m}.
    \end{align*}
We conclude that $U_n + U_{n + \frac{\pi_U}{2}} \equiv 0 \pmod{m}$, as desired.
\end{proof}

We now give the $\LucasFormalVSeq$ variation of the previous theorem. Observe that although the two theorems seems analogous, the $\LucasFormalVSeq$ version has more restrictions due to the fact that the entry point of $m$ in $\LucasFormalVSeq$ is not always defined. Morever, the proof employs results that require $p$ to be odd, or $p$ to be even and $m$ be odd, as per Convention~\ref{conv:restriction_of_p_and_m_for_equality_of_periods}.

\begin{theorem}\label{thm:antipodal_result_for_V_n_sequence_with_order_4}
Let $m > 2$ be given such that $e_V(m)$ is defined, and set $V_n:=\LucasFormalVTerm$ with $q=-1$. Assume that $p$ is odd, or $p$ is even and $m$ is odd. Then for all $n \in \mathbb{Z}$, we have
\begin{align*}
    \omega_V(m) = 4 &\;\implies\; V_{n} + V_{n+\frac{\pi_V(m)}{2}} \equiv 0 \pmodd{m}.
\end{align*}
\end{theorem}

\begin{proof}
Let $m > 2$ be given, and assume that $\omega_V = 4$. Then we have $\frac{\pi_V}{e_V}=4$ and hence $\frac{\pi_V}{2} = 2e_V$. It follows that
    \begin{align*}
        V_n + V_{n + \frac{\pi_V}{2}} &= V_n + V_{n + 2e_V}\\
        &= V_n + \left( V_{n-1} U_{2e_V} + V_{n} U_{2e_V + 1} \right) &\text{by Identity~\eqref{eq:Lucas_identity_Vorobiev_V_n_like_result} of Lemma~\ref{lem:U_n_Lucas_sequence_identities_part_1}}\\
        &= V_n + \left( V_{n-1} U_{e_U} + V_{n} U_{e_U + 1} \right) &\text{since $e_U = 2 e_V$ by Theorem~\ref{thm:Jurgen_conjecture_on_existence_of_entry_point_generalized_to_Un_Vn_setting}}\\
        &\equiv V_n + V_{n} U_{e_U + 1} \pmodd{m} &\text{since $U_{e_U} \equiv 0 \pmodd{m}$}\\
        &\equiv V_n + V_n \cdot (-1) \pmodd{m}\\
        &\equiv 0 \pmodd{m},
    \end{align*}
 where the second to last congruence holds since $\omega_V = 4$ implies that $e_V$ is even by Theorem~\ref{thm:omega_V_equals_4}, and hence we have $U_{e_U + 1} \equiv -1 \pmod{m}$ by Theorem~\ref{thm:consequence_of_existence_of_even_entry_point_for_V_n}. We conclude that $V_n + V_{n + \frac{\pi_V}{2}} \equiv 0 \pmod{m}$, as desired.
\end{proof}

\begin{example}\label{exam:Fib_Luc}
For $\LucasFormalUSeq$ and $\LucasFormalVSeq$ with $p=1$ and $q=-1$, we have the well-known sequences $\FibSeq$ and $\LucSeq$. In Figure~\ref{fig:Fib_and_Luc_circles_with_pi_over_e_equal_4}, we give two examples of these sequences modulo $m$ which fulfill the hypotheses of Theorems~\ref{thm:antipodal_result_for_U_n_sequence_with_order_4} and \ref{thm:antipodal_result_for_V_n_sequence_with_order_4}, respectively.
\begin{figure}[H]
    \centering
    \setcounter{ga}{1}
    \begin{tikzpicture}[scale=0.56,every node/.style={scale=0.56}]
        \def \n {28}
        \def \radius {5}
        \def \innerRadius {4.2}
        \draw [dashed, blue] (-\innerRadius,0) -- (\innerRadius,0);
        \draw [dashed, blue] (0,-\innerRadius) -- (0,\innerRadius);
        \node at (0,0) [draw=red, fill=white] {\LARGE$\bm{m=13}$};
        \draw circle(\radius)
              foreach\s in{1,...,\n}{
                  (-360/\n*\s-90:-\radius)node[fill=red,circle, inner sep=2pt]{}
                  node[anchor=-360/\n*\s-90]{\large\blue{$\widetilde{F}_{\s}\ifnum\s=\n\relax=\widetilde{F}_0\fi \ifnum\s=7\relax=\widetilde{F}_e\fi \ifnum\s=21\relax=\widetilde{F}_{3e}\fi$}}
              };
        \draw foreach\v in {1, 1, 2, 3, 5, 8, 0, 8, 8, 3, 11, 1, 12, 0, 12, 12, 11, 10, 8, 5, 0, 5, 5, 10, 2, 12, 1, 0}{
                  (-360/\n*\thega-89:-\innerRadius)
                  node[anchor=-360/\n*\thega-89]{\large{$\mathbf{\v}$}\stepcounter{ga}}
              };
    \end{tikzpicture}
    \setcounter{ga}{1}
    \begin{tikzpicture}[scale=0.56,every node/.style={scale=0.56}]
        \def \n {24}
        \def \radius {5}
        \def \innerRadius {4.2}
        \draw [dashed, blue] (-\innerRadius,0) -- (\innerRadius,0);
        \draw [dashed, blue] (0,-\innerRadius) -- (0,\innerRadius);
        \node at (0,0) [draw=red, fill=white] {\LARGE$\bm{m=9}$};
        \draw circle(\radius)
              foreach\s in{1,...,\n}{
                  (-360/\n*\s-90:-\radius)node[fill=red,circle, inner sep=2pt]{}
                  node[anchor=-360/\n*\s-90]{\large\blue{$\widetilde{L}_{\s}\ifnum\s=\n\relax=\widetilde{L}_0\fi \ifnum\s=6\relax=\widetilde{L}_e\fi \ifnum\s=18\relax=\widetilde{L}_{3e}\fi$}}
              };
        \draw foreach\v in {1, 3, 4, 7, 2, 0, 2, 2, 4, 6, 1, 7, 8, 6, 5, 2, 7, 0, 7, 7, 5, 3, 8, 2}{
                  (-360/\n*\thega-89:-\innerRadius)
                  node[anchor=-360/\n*\thega-89]{\large{$\mathbf{\v}$}\stepcounter{ga}}
              };
    \end{tikzpicture}
    \caption{$\FibSeq$ modulo $m=13$ with $\pi_F=28$, and $\LucSeq$ modulo $m=9$ with $\pi_L=24$}
    \label{fig:Fib_and_Luc_circles_with_pi_over_e_equal_4}
\end{figure}
\end{example}

\begin{example}\label{exam:Pell_Qell}
For $\LucasFormalUSeq$ and $\LucasFormalVSeq$ with $p=2$ and $q=-1$, we have the well-known sequences $\PellSeq$ and $\left(2 Q_n\right)_{n \geq 0}$, where each term in the second sequence is exactly twice the value of the corresponding term in our associated Pell sequence $\QellSeq$ given in Definition~\ref{def:Pell_Qell_numbers}. However, despite the fact that $V_n(2,1) = 2Q_n$ for all $n \geq 0$, the result of Theorem~\ref{thm:antipodal_result_for_V_n_sequence_with_order_4} still holds for $\QellSeq$. This is because of the fact that $m>2$ divides $Q_n$ implies $m$ divides $V_n(2,-1)$, the fact that $Q_n$ is odd for all $n\geq 0$, and the following assertions whose justifications we leave to the interested reader (below $V$ refers to the sequence $V_n(2,-1)$):
\begin{enumerate}[(i),leftmargin=1.1in]
    \item If $e_Q(m)$ exists, then $e_V(m)$ also exists and $e_Q(m)=e_V(m)$.
    \item If $\omega_Q(m)$ exists, then $\omega_V(m)$ also exists and $\omega_Q(m)=\omega_V(m)$.
\end{enumerate}
In Figure~\ref{fig:Pell_and_Qell_circles_with_pi_over_e_equal_4}, we give two examples of these sequences $\PellSeq$ and $\QellSeq$ modulo $m$ which fulfill the hypotheses of Theorems~\ref{thm:antipodal_result_for_U_n_sequence_with_order_4} and \ref{thm:antipodal_result_for_V_n_sequence_with_order_4}, respectively.
\begin{figure}[H]
    \centering
    \setcounter{ga}{1}
    \begin{tikzpicture}[scale=0.56,every node/.style={scale=0.56}]
        \def \n {20}
        \def \radius {5}
        \def \innerRadius {4.2}
        \draw [dashed, blue] (-\innerRadius,0) -- (\innerRadius,0);
        \draw [dashed, blue] (0,-\innerRadius) -- (0,\innerRadius);
        \node at (0,0) [draw=red, fill=white] {\LARGE$\bm{m=29}$};
        \draw circle(\radius)
              foreach\s in{1,...,\n}{
                  (-360/\n*\s-90:-\radius)node[fill=red,circle, inner sep=2pt]{}
                  node[anchor=-360/\n*\s-90]{\large\blue{$\widetilde{P}_{\s}\ifnum\s=\n\relax=\widetilde{P}_0\fi \ifnum\s=5\relax=\widetilde{P}_e\fi \ifnum\s=15\relax=\widetilde{P}_{3e}\fi$}}
              };
        \draw foreach\v in {1, 2, 5, 12, 0, 12, 24, 2, 28, 0, 28, 27, 24, 17, 0, 17, 5, 27, 1, 0}{
                  (-360/\n*\thega-89:-\innerRadius)
                  node[anchor=-360/\n*\thega-89]{\large{$\mathbf{\v}$}\stepcounter{ga}}
              };
    \end{tikzpicture}
    \setcounter{ga}{1}
    \begin{tikzpicture}[scale=0.56,every node/.style={scale=0.56}]
        \def \n {24}
        \def \radius {5}
        \def \innerRadius {4.2}
        \draw [dashed, blue] (-\innerRadius,0) -- (\innerRadius,0);
        \draw [dashed, blue] (0,-\innerRadius) -- (0,\innerRadius);
        \node at (0,0) [draw=red, fill=white] {\LARGE$\bm{m=33}$};
        \draw circle(\radius)
              foreach\s in{1,...,\n}{
                  (-360/\n*\s-90:-\radius)node[fill=red,circle, inner sep=2pt]{}
                  node[anchor=-360/\n*\s-90]{\large\blue{$\widetilde{Q}_{\s}\ifnum\s=\n\relax=\widetilde{Q}_0\fi \ifnum\s=6\relax=\widetilde{Q}_e\fi \ifnum\s=18\relax=\widetilde{Q}_{3e}\fi$}}
              };
        \draw foreach\v in {1, 3, 7, 17, 8, 0, 8, 16, 7, 30, 1, 32, 32, 30, 26, 16, 25, 0, 25, 17, 26, 3, 32, 1}{
                  (-360/\n*\thega-89:-\innerRadius)
                  node[anchor=-360/\n*\thega-89]{\large{$\mathbf{\v}$}\stepcounter{ga}}
              };
    \end{tikzpicture}
    \caption{$\PellSeq$ modulo $m=29$ with $\pi_P=20$, and $\QellSeq$ modulo $m=33$ with $\pi_Q=24$}
    \label{fig:Pell_and_Qell_circles_with_pi_over_e_equal_4}
\end{figure}
\end{example}



\subsection{Additive inverse palindromes in \texorpdfstring{$\LucasFormalUSeq$}{Un} when \texorpdfstring{$q = 1$}{q is 1}}\label{subsec:additiv_inverse_palindrome}

In the balancing sequence $\BalSeq$, we observed a pattern in the fundamental period for every moduli $m>2$. For example, below we give the period of $\BalSeq$ modulo 8.
\begin{center}
\begin{tikzpicture}[scale=0.75,every node/.style={scale=0.75}]
\definecolor{canaryyellow}{rgb}{1.0, 0.94, 0.0}
\definecolor{daffodil}{rgb}{1.0, 1.0, 0.19}
\definecolor{electricblue}{rgb}{0.49, 0.98, 1.0}
\definecolor{green-yellow}{rgb}{0.68, 1.0, 0.18}
  \SetGraphUnit{2}
  \GraphInit[vstyle = Shade]
\tikzset{
  VertexStyle/.style = {shape = circle,
ball color = daffodil,
font = \large\bfseries,
text = black,
inner sep = 2pt,
outer sep = 0pt,
minimum size = 22 pt},
  EdgeStyle/.append style = {stealth-stealth, bend right} }
  \node[VertexStyle](A){0};
  \node[VertexStyle,right=of A](B){1};
  \node[VertexStyle,right=of B](C){6};
  \node[VertexStyle,right=of C](D){3};
  \node[VertexStyle,right=of D](E){4};
  \node[VertexStyle,right=of E](F){5};
  \node[VertexStyle,right=of F](G){2};
  \node[VertexStyle,right=of G](H){7};
  \node[VertexStyle,right=of H](I){0};
  
  \node[text=red, above=.2cm of A]{0};
  \node[text=red, above=.2cm of B]{1};
  \node[text=red, above=.2cm of C]{6};
  \node[text=red, above=.2cm of D]{35};
  \node[text=red, above=.2cm of E]{204};
  \node[text=red, above=.2cm of F]{1189};
  \node[text=red, above=.2cm of G]{6930};
  \node[text=red, above=.2cm of H]{40391};
  \node[text=red, above=.2cm of I]{235416};

  \node[text=blue, above=.7cm of A]{$B_0$};
  \node[text=blue, above=.7cm of B]{$B_1$};
  \node[text=blue, above=.7cm of C]{$B_2$};
  \node[text=blue, above=.7cm of D]{$B_3$};
  \node[text=blue, above=.7cm of E]{$B_4$};
  \node[text=blue, above=.7cm of F]{$B_5$};
  \node[text=blue, above=.7cm of G]{$B_6$};
  \node[text=blue, above=.7cm of H]{$B_7$};
  \node[text=blue, above=.7cm of I]{$B_8$};
  
  \Edge(A)(I)
  \Edge(B)(H)
  \Edge(C)(G)
  \Edge(D)(F)
  \Loop[dist=2cm, dir=WE, style={stealth-stealth}](E.south)
\end{tikzpicture}
\end{center}
\vspace{-.25in}
It is clear that vertices connected by arrows are additive inverses of each other modulo 8. We call this an additive inverse palindrome defined in the $\LucasFormalUSeq$ setting as follows.
\begin{definition}\label{def:additive_inverse_palindrome}
For $m>2$, we say $\LucasFormalUSeq \pmodd{m}$ is an \textit{additive inverse palindrome} if $U_n + U_{\pi_U(m) - n} \equiv 0 \pmod{m}$ for all $0 \leq n \leq \lfloor \frac{\pi_U(m)}{2} \rfloor$.
\end{definition}

It turns out this behavior is not unique to $\BalSeq$. In fact, every sequence $\LucasFormalUSeq$ modulo $m>2$ with $q = 1$ is an additive inverse palindrome. We prove this in the following theorem.

\begin{theorem}\label{thm:Un_with_q_equal_1_ additive_inverse_palindrome}
Let $m > 2$ be given and set $U_n:=\LucasFormalUTerm$ with $q=1$. Then the fundamental period is an additive inverse palindrome. More precisely, for all $0 \leq n \leq \lfloor \frac{\pi_U(m)}{2} \rfloor$, we have
$$ U_n + U_{\pi_U(m) - n} \equiv 0 \pmodd{m}.$$
\end{theorem}

\begin{proof}
Let $m > 2$ be given. We prove the result by induction on $n$. For $n=0$, we have $U_0 + U_{\pi_U} \equiv 0 \pmod{m}$. By the recurrence relation for $\LucasUSeq$ with $q=1$, we have $U_1 = p \cdot U_0 - U_{-1}$ implying $U_1 + U_{-1} = 0$. Thus for $n=1$, we have $U_1 + U_{\pi_U-1} \equiv 0 \pmod{m}$. Hence the base cases hold. Now suppose that the result holds for $k-1$ and $k$ for some $k \geq 1$, and so we have the following:
\begin{align}
    U_{k-1} + U_{\pi_U-(k-1)} &\equiv 0 \pmodd{m} \label{eq:U_additive_inverse_palindrome_1}\\
    U_k + U_{\pi_U-k} &\equiv 0 \pmodd{m}. \label{eq:U_additive_inverse_palindrome_2}
\end{align}
It suffices to show that $ U_{k+1} + U_{\pi_U-(k+1)} \equiv 0 \pmod{m}$. To that end, observe that
\begin{align*}
    U_{k+1}+U_{\pi_U-(k+1)} &\equiv \left(p \cdot U_k - U_{k-1}\right) + \left(p \cdot U_{\pi_U-k} - U_{\pi_U-(k-1)}\right)  \pmodd{m}\\
    &\equiv p \cdot \left( U_k + U_{\pi_U - k}\right) - \left( U_{k-1} + U_{\pi_U - (k-1)} \right) \pmodd{m}\\
    &\equiv 0 \pmodd{m},
\end{align*}
where the first congruence holds since the $\LucasUSeq$ recurrence yields
$$U_{\pi_U - (k-1)} \equiv p \cdot U_{\pi_U - k} - U_{\pi_U - (k+1)} \pmodd{m},$$
which implies $p \cdot U_{\pi_U - k} - U_{\pi_U - (k-1)} \equiv U_{\pi_U - (k+1)} \pmod{m}$, and the third congruence holds by Identities~\eqref{eq:U_additive_inverse_palindrome_1} and \eqref{eq:U_additive_inverse_palindrome_2}. Thus the result holds.
\end{proof}



\subsection{Double palindromes in \texorpdfstring{$\LucasFormalUSeq$}{Un} when \texorpdfstring{$q = 1$}{q is 1} and \texorpdfstring{$\omega_U(m) = 2$}{order is 2}}

Recall by Corollary~\ref{cor:Order_in_Un_when_q_equals_1_or_minus_1}, we have $\omega_U(m) \in \{1,2\}$ when $q=1$. In $\BalSeq$, we observed that we sometimes get a special additive inverse palindrome when $\omega_B(m)=2$. We exhibit this behavior in the following fundamental period of $\BalSeq$ modulo 17.
\begin{center}
\begin{tikzpicture}[scale=0.9,every node/.style={scale=0.9}]
\definecolor{canaryyellow}{rgb}{1.0, 0.94, 0.0}
\definecolor{daffodil}{rgb}{1.0, 1.0, 0.19}
\definecolor{electricblue}{rgb}{0.49, 0.98, 1.0}
\definecolor{green-yellow}{rgb}{0.68, 1.0, 0.18}
  \SetGraphUnit{2}
  \GraphInit[vstyle = Shade]
\tikzset{
  VertexStyle/.style = {shape = circle,
ball color = daffodil,
font = \large\bfseries,
text = black,
inner sep = 2pt,
outer sep = 0pt,
minimum size = 22pt},
  EdgeStyle/.append style = {stealth-stealth, bend right} }
  \begin{scope}[scale=1, transform shape]
  \node[VertexStyle](A){0};
  \node[VertexStyle,right=of A](B){1};
  \node[VertexStyle,right=of B](C){6};
  \node[VertexStyle,right=of C](D){1};
  \node[VertexStyle,right=of D](E){0};
  \node[VertexStyle,right=of E](F){16};
  \node[VertexStyle,right=of F](G){11};
  \node[VertexStyle,right=of G](H){16};
  \node[VertexStyle,right=of H](I){0};
  
  \node[text=red, above=.2cm of A]{0};
  \node[text=red, above=.2cm of B]{1};
  \node[text=red, above=.2cm of C]{6};
  \node[text=red, above=.2cm of D]{35};
  \node[text=red, above=.2cm of E]{204};
  \node[text=red, above=.2cm of F]{1189};
  \node[text=red, above=.2cm of G]{6930};
  \node[text=red, above=.2cm of H]{40391};
  \node[text=red, above=.2cm of I]{235416};

  \node[text=blue, above=.7cm of A]{$B_0$};
  \node[text=blue, above=.7cm of B]{$B_1$};
  \node[text=blue, above=.7cm of C]{$B_2$};
  \node[text=blue, above=.7cm of D]{$B_3$};
  \node[text=blue, above=.7cm of E]{$B_4$};
  \node[text=blue, above=.7cm of F]{$B_5$};
  \node[text=blue, above=.7cm of G]{$B_6$};
  \node[text=blue, above=.7cm of H]{$B_7$};
  \node[text=blue, above=.7cm of I]{$B_8$};
  
  \Edge(A)(E)
  \Edge(B)(D)
  \Loop[dist=2cm, dir=WE, style={stealth-stealth}](C.south)
  \Edge(E)(I)
  \Edge(F)(H)
  \Loop[dist=2cm, dir=WE, style={stealth-stealth}](G.south)
  \end{scope}
\end{tikzpicture}
\end{center}
It is clear that vertices connected by arrows are equal to each other. Since we have two consecutive palindromes in the fundamental period, we call this a double palindrome defined in the $\LucasFormalUSeq$ setting as follows.
\begin{definition}\label{def:double_palindrome}
For $m>2$, we say $\LucasFormalUSeq \pmodd{m}$ with order $\omega_U(m) = 2$ is a \textit{double palindrome} if $U_n \equiv U_{e_U(m)-n} \pmod{m}$ and $U_{e_U(m)+n} \equiv U_{\pi_U(m)-n} \pmod{m}$ for all $0 \leq n \leq \lfloor \frac{e_U(m)}{2} \rfloor$.
\end{definition}

Just as in the additive inverse palindromes, it turns out this behavior is not unique to $\BalSeq$. In fact, every sequence $\LucasFormalUSeq$ modulo $m>2$ with $q = 1$ and $\omega_U(m) = 2$ is a double palindrome if the condition $U_{e_U(m)+1} \equiv -1 \pmod{m}$ holds. We prove this in the following theorem.

\begin{theorem}\label{thm:U_n_with_q_equal_1_double_palindrome}
Let $m > 2$ be given and set $U_n:=\LucasFormalUTerm$ with $q=1$ and $\omega_U(m) = 2$. If $U_{e_U(m)+1} \equiv -1 \pmod{m}$, then the fundamental period is a double palindrome. More precisely, for all $0 \leq n \leq \lfloor \frac{e_U(m)}{2} \rfloor$, we have
\begin{align}
   U_n &\equiv U_{e_U(m)-n} \pmodd{m} \label{eq:U_rightside_palindrome}\\
   U_{e_U(m)+n} &\equiv U_{\pi_U(m)-n} \pmodd{m}.\label{eq:U_leftside_palindrome}
\end{align}
\end{theorem}

\begin{proof}
Let $m>2$ be given, and assume that $U_{e_U+1} \equiv -1 \pmod{m}$. We prove Identity~\eqref{eq:U_rightside_palindrome} holds by induction on $n$, and consequently this will yield Identity~\eqref{eq:U_leftside_palindrome} by Theorem~\ref{thm:Un_with_q_equal_1_ additive_inverse_palindrome}. For $n=0$, we have $U_0 \equiv U_{e_U} \pmod{m}$ since $U_0 \equiv 0 \equiv U_{e_U} \pmod{m}$. By the recurrence relation for $\LucasUSeq$ with $q=1$, it follows that $U_{e_U+1} = p \cdot U_{e_U} - U_{e_U-1}$ and hence $U_{e_U-1} \equiv 1 \pmod{m}$ since $U_{e_U+1} \equiv -1 \pmod{m}$ by assumption. Thus for $n=1$, we have $U_1 \equiv U_{e_U - 1} \pmod{m}$ since $U_1 = 1$. Hence the base cases hold. Now suppose that the result holds for $k-1$ and $k$ for some $k \geq 1$, and so we have the following:
\begin{align}
    U_{k-1} &\equiv U_{e_U - (k-1)} \pmodd{m}\label{eq:U_rightside_1}\\
    U_k &\equiv U_{e_U - k} \pmodd{m}.\label{eq:U_rightside_2}
\end{align}
It suffices to show that $U_{k+1} \equiv U_{e_U - (k+1)}  \pmod{m}$. To that end, observe that
\begin{align*}
    U_{k+1} &= p \cdot U_k - U_{k-1} &\text{by the $\LucasUSeq$ recurrence}\\
    &\equiv p \cdot U_{e_U-k} - U_{e_U-(k-1)} \pmodd{m} &\text{by Identities~\eqref{eq:U_rightside_1} and \eqref{eq:U_rightside_2}}\\
    &\equiv U_{e_U-(k+1)} \pmodd{m},
\end{align*}
which holds since the $\LucasUSeq$ recurrence yields $U_{e_U - (k-1)} \equiv p \cdot U_{e_U - k} - U_{e_U - (k+1)} \pmod{m}$ implying that $p \cdot U_{e_U - k} - U_{e_U - (k-1)} \equiv U_{e_U - (k+1)} \pmod{m}$. We conclude that Identity~\eqref{eq:U_rightside_palindrome} holds. Lastly, since the sequence terms $U_0, U_1, \ldots, U_{\pi_U}$ modulo $m$ form an additive inverse palindrome by Theorem~\ref{thm:Un_with_q_equal_1_ additive_inverse_palindrome}, we conclude that Identity~\eqref{eq:U_leftside_palindrome} also holds.
\end{proof}

\begin{remark}
For the sequences $\LucasFormalUSeq$ with $q=1$ and $\omega_U(m) = 2$, the extra condition that $U_{e_U+1} \equiv -1 \pmod{m}$ holds is indeed vital to producing a double palindrome as the following fundamental period shows for the balancing sequence $\BalSeq$ modulo 51:
\begin{center}
\begin{tikzpicture}[scale=0.75,every node/.style={scale=0.75}]
\definecolor{canaryyellow}{rgb}{1.0, 0.94, 0.0}
\definecolor{daffodil}{rgb}{1.0, 1.0, 0.19}
\definecolor{electricblue}{rgb}{0.49, 0.98, 1.0}
\definecolor{green-yellow}{rgb}{0.68, 1.0, 0.18}
  \SetGraphUnit{2}
  \GraphInit[vstyle = Shade]
\tikzset{
  VertexStyle/.style = {shape = circle,
ball color = daffodil,
font = \large\bfseries,
text = black,
inner sep = 2pt,
outer sep = 0pt,
minimum size = 22 pt},
  EdgeStyle/.append style = {stealth-stealth, bend right} }
  \node[VertexStyle](A){0};
  \node[VertexStyle,right=of A](B){1};
  \node[VertexStyle,right=of B](C){6};
  \node[VertexStyle,right=of C](D){35};
  \node[VertexStyle,right=of D](E){0};
  \node[VertexStyle,right=of E](F){16};
  \node[VertexStyle,right=of F](G){45};
  \node[VertexStyle,right=of G](H){50};
  \node[VertexStyle,right=of H](I){0};
  
  \node[text=red, above=.2cm of A]{0};
  \node[text=red, above=.2cm of B]{1};
  \node[text=red, above=.2cm of C]{6};
  \node[text=red, above=.2cm of D]{35};
  \node[text=red, above=.2cm of E]{204};
  \node[text=red, above=.2cm of F]{1189};
  \node[text=red, above=.2cm of G]{6930};
  \node[text=red, above=.2cm of H]{40391};
  \node[text=red, above=.2cm of I]{235416};

  \node[text=blue, above=.7cm of A]{$B_0$};
  \node[text=blue, above=.7cm of B]{$B_1$};
  \node[text=blue, above=.7cm of C]{$B_2$};
  \node[text=blue, above=.7cm of D]{$B_3$};
  \node[text=blue, above=.7cm of E]{$B_4$};
  \node[text=blue, above=.7cm of F]{$B_5$};
  \node[text=blue, above=.7cm of G]{$B_6$};
  \node[text=blue, above=.7cm of H]{$B_7$};
  \node[text=blue, above=.7cm of I]{$B_8$};
  
  \Edge(A)(I)
  \Edge(B)(H)
  \Edge(C)(G)
  \Edge(D)(F)
  \Loop[dist=2cm, dir=WE, style={stealth-stealth}](E.south)
\end{tikzpicture}
\end{center}
\vspace{-.25in}
It is interesting to note that while $B_{e_B+1} \not\equiv -1 \pmod{51}$, the value $B_{e_B+1}$ is congruent to 16, and the (multiplicative) order of 16 modulo 51 is 2, just as the order of $-1$ is. However, this is no surprise since the we know 16 is the ``multiplier'' in this case (see Convention~\ref{conv:the_multiplier}).
\end{remark}



\subsection{True palindromes in \texorpdfstring{$\LucasFormalVSeq$}{Vn} when \texorpdfstring{$q = 1$}{q is 1}}

For the five sequences ($\FibSeq$, $\LucSeq$, $\PellSeq$, $\QellSeq$, and $\BalSeq$), we gave specific $m$ values where the fundamental periods of these five sequences were not true palindromes in Example~\ref{exam:Fib_Luc}, Example~\ref{exam:Pell_Qell}, and Subsection~\ref{subsec:additiv_inverse_palindrome}, respectively. However, unlike these five sequences, the fundamental period of the Lucas-balancing sequence $\LucBalSeq$ modulo $m$ yields true palindromes for all $m > 2$.
In Figure~\ref{fig:Lucas_balancing_modulo_57}, we place the fundamental period of $\LucBalSeq$ modulo $57$ in a circle.

\begin{figure}[h]
    \centering
    \setcounter{ga}{1}
    \begin{tikzpicture}[scale=0.56,every node/.style={scale=0.56}]
        \def \n {20}
        \def \radius {5}
        \def \innerRadius {4.2}
        \draw [dashed, blue] (-\innerRadius,0) -- (\innerRadius,0);
        \draw [dashed, blue] (0,-\innerRadius) -- (0,\innerRadius);
        \node at (0,0) [draw=red, fill=white] {\LARGE$\bm{m=57}$};
        \draw circle(\radius)
              foreach\s in{1,...,\n}{
                  (-360/\n*\s-90:-\radius)node[fill=red,circle, inner sep=2pt]{}
                  node[anchor=-360/\n*\s-90]{\large\blue{$\widetilde{C}_{\s}\ifnum\s=\n\relax=\widetilde{C}_0\fi \ifnum\s=5\relax=\widetilde{C}_e\fi \ifnum\s=15\relax=\widetilde{C}_{3e}\fi$}}
              };
        \draw foreach\v in {3, 17, 42, 7, 0, 50, 15, 40, 54, 56, 54, 40, 15, 50, 0, 7, 42, 17, 3, 1}{
                  (-360/\n*\thega-89:-\innerRadius)
                  node[anchor=-360/\n*\thega-89]{\large{$\mathbf{\v}$}\stepcounter{ga}}
              };
    \end{tikzpicture}
    \caption{$\LucBalSeq$ modulo $m=57$ with $\pi_C=20$ and $e_C=5$}
    \label{fig:Lucas_balancing_modulo_57}
\end{figure}

By examination of the circle in Figure~\ref{fig:Lucas_balancing_modulo_57}, we see that not only does the fundamental period form a palindrome, but we also can see that the terms which are equidistant from any fixed zero sum to 0 modulo 57 (that is, they are additive inverses of each other). It turns out these behaviors hold in the more general setting of any Lucas sequence $\LucasFormalVSeq$ modulo $m$ with $q = 1$. More precisely, we have the following for $m>2$:
\begin{enumerate}[(i)]
    \item Palindromes exist in every fundamental period for $\LucasFormalVSeq$ when $q=1$ (see Theorem~\ref{thm:V_n_with_q_equal_1_palindrome_result}).
    \item If $\omega_V(m) = 4$ and $q=1$, then the terms equidistant from any fixed zero of the fundamental period are additive inverses of each other (see Theorem~\ref{thm:V_n_order_4_vertical_slice_result}).
\end{enumerate}

\begin{theorem}\label{thm:V_n_with_q_equal_1_palindrome_result}
Let $m > 2$ be given and set $V_n:=\LucasFormalVTerm$ with $q=1$. Then the fundamental period is a palindrome. More precisely, for all $0 \leq n \leq \pi_V(m)$, we have
$$V_n \equiv V_{\pi_V(m) - n} \pmodd{m}.$$
\end{theorem}

\begin{proof}
Let $m>2$ be given. We prove the result by induction on $n$. For $n=0$, we have $V_0 \equiv 2 \equiv V_{\pi_V} \pmod{m}$. By the recurrence relation for $\LucasVSeq$ with $q=1$, it follows that $V_1 = p \cdot V_0 - V_{-1}$ and hence $V_{-1} = p$ since $V_0 = 2$ and $V_1 = p$. Therefore for $n=1$, we have $V_1 \equiv p \equiv V_{\pi_V-1} \pmod{m}$. Hence the base cases hold. Now suppose that the result holds for $k-1$ and $k$ for some $k \geq 1$, and so we have the following:
\begin{align}
    V_{k-1} &\equiv V_{\pi_V-(k-1)} \pmodd{m} \label{eq:V_palindrome_1}\\
    V_k &\equiv V_{\pi_V-k} \pmodd{m}. \label{eq:V_palindrome_2}
\end{align}
It suffices to show that $ V_{k+1} \equiv V_{\pi_V-(k+1)} \pmod{m}$. To that end, observe that
\begin{align*}
    V_{k+1} &= p \cdot V_k - V_{k-1} &\text{by the $\LucasVSeq$ recurrence}\\
    &\equiv p \cdot V_{\pi_V-k} - V_{\pi_V-(k-1)} \pmodd{m} &\text{by Identities~\eqref{eq:V_palindrome_1} and \eqref{eq:V_palindrome_2}}\\
    &\equiv V_{\pi_V-(k+1)} \pmodd{m},
\end{align*}
which holds since the $\LucasVSeq$ recurrence yields $V_{\pi_V - (k-1)} \equiv p \cdot V_{\pi_V - k} - V_{\pi_V - (k+1)} \pmod{m}$ implying that $p \cdot V_{\pi_V - k} - V_{\pi_V - (k-1)} \equiv V_{\pi_V - (k+1)} \pmod{m}$. Thus the result follows.
\end{proof}

\begin{theorem}\label{thm:V_n_order_4_vertical_slice_result}
Let $m > 2$ be given and set $V_n:=\LucasFormalVTerm$ with $q=1$. If $\omega_V(m) = 4$, then the terms in the fundamental period equidistant from any fixed zero of the sequence sum to $m$. More precisely, for all $0 \leq n \leq e_V(m)$, we have
\begin{align}
   V_{e_V(m)+n} + V_{e_V(m)-n} &\equiv 0 \pmodd{m} \label{eq:V_e_vertical}\\
   V_{3e_V(m)+n} + V_{3e_V(m)-n} &\equiv 0 \pmodd{m}.\label{eq:V_3e_vertical}
\end{align}
\end{theorem}

\begin{proof}
Let $m>2$ be given. Observe that
$$ V_{e_V + n} + V_{e_V - n} \equiv V_{e_V + 1} U_n + V_{e_V + 1} U_{-n} \equiv V_{e_V + 1}\left( U_n + U_{\pi_U - n} \right) \equiv 0 \pmodd{m},$$
where the first congruence holds by Identity~\eqref{eq:Ballot_inspired_congruence} of Theorem~\ref{thm:Christian_Ballot_equality_of_periods}, and the second congruence holds by Theorem~\ref{thm:Un_with_q_equal_1_ additive_inverse_palindrome}. We conclude that Identity~\eqref{eq:V_e_vertical} holds. Since the sequence terms $V_0, V_1, \ldots, V_{\pi_V}$ modulo $m$ form a true palindrome by Theorem~\ref{thm:V_n_with_q_equal_1_palindrome_result}, we conclude that Identity~\eqref{eq:V_3e_vertical} also holds. Thus the result follows.
\end{proof}

\begin{remark}\label{rem:when_does_vertical_slice_result_hold_for_q_equal_1}
Upon observing a multitude of fundamental periods as we fix $q=1$ and let $p$ vary for $\LucasFormalVSeq$ modulo $m$, we see that there are many instances of the entry point not existing (and hence $\omega_V(m)$ also not existing), but yet the ``vertical slice result'' still holds. That is, in the graphical setting, if we draw a vertical line through the circle such that the line passes through the location of two terms of the fundamental period, then those two terms are additive inverses of one another. See Question~\ref{question:vertical_slice_result_in_q_equal_1_setting}.
\end{remark}



\section{Open questions and future work}\label{sec:open questions}

\begin{question}\label{question:aBa_Diego_Oliver}
In the Problem Session at the 21st International Fibonacci Conference, author Mbirika proposed two statements to prove, one of which was
\begin{align}
    U_{2n} \equiv 0 \pmodd{m} \text{ and } U_{2n+1} \equiv q^n \pmodd{m} &\;\Longleftrightarrow\; U_n \equiv 0 \pmodd{m} \label{eq:Desmond_lemma_1_faulty_biconditional}
\end{align}
It turns out that the sufficiency condition is incorrect and holds only when either $p$ or $m$ is odd. That is now Implication~\eqref{eq:Desmond_lemma_1_left_to_right} of Lemma~\ref{lem:Desmond_lemma_1_generalization} in this paper. After the problem session, conference participants Diego Garcia-Fernandezsesma (\href{mailto:dipianad@gmail.com}{dipianad@gmail.com}) and Oliver Lippard (\href{mailto:hlippard@charlotte.edu}{hlippard@charlotte.edu}) consulted with Mbirika, and the three of them began to formulate a new way to think about Implication~\eqref{eq:Desmond_lemma_1_left_to_right} of Lemma~\ref{lem:Desmond_lemma_1_generalization} when both $p$ and $m$ are even. In particular, they observed that when both $p$ and $m$ are even, the conclusion of Implication~\eqref{eq:Desmond_lemma_1_left_to_right} is either $U_n \equiv 0 \pmod{m}$ or $U_n \equiv \frac{m}{2} \pmod{m}$. In particular, Lippard further observed that when $U_n \equiv \frac{m}{2} \pmod{m}$ occurs, the index $n$ is of the form $c \cdot \frac{e_U(m)}{2}$, where $c$ is some odd integer. These observations led to the following conjecture:

\begin{conjecture}\label{conj:aBa_Diego_Oliver}
Set $U_n:=U_n(p,q)$, and let $m \geq 1$ and $n \in \mathbb{Z}$. Assume that both $p$ and $m$ are even. If $U_{2n} \equiv 0 \pmod{m}$ and $U_{2n+1} \equiv q^n \pmod{m}$, then exactly one of the following two conclusions occur:
\begin{enumerate}[(i)]
    \item $U_n \equiv 0 \pmod{m}$, or
    \item $U_n \equiv \frac{m}{2} \pmod{m}$ where $n = c \cdot \frac{e_U(m)}{2}$, where $c$ is some odd integer and $e_U(m)$ is the entry point of $m$ in $(U_n)_{n \geq 0}$.
\end{enumerate}
\end{conjecture}
\noindent Can we give conditions for when Conclusion (i) versus Conclusion (ii) occurs?

\end{question}

\begin{question}\label{question:Ballot_extended}
Subsection~\ref{subsec:consequences_of_existence_of_entry_point_in_V_n} begins with a convention restricting the $p$ and $m$ values to either of the following classes: (1) $p$ is odd, or (2) $p$ is even and $m$ is odd. This is due to the fact that those were the conditions of the sufficiency criteria for the existence of $e_V(m)$ to guarantee that $\pi_U(m) = \pi_V(m)$ in Corollary~\ref{cor:Sufficiency_condition_for_equality_of_periods}. However, data generated with \texttt{Mathematica} provides ample support that when $q = -1$, the corollary holds even in the setting when both $p$ and $m$ are even. This remains an open problem.
\end{question}

\begin{question}
Related to Question~\ref{question:Ballot_extended} and from data generated with \texttt{Mathematica}, we find that the result in Corollary~\ref{cor:Sufficiency_condition_for_equality_of_periods} almost always holds in the $q = 1$ setting when both $p$ and $m$ are even. For example, when $p \equiv 0 \pmod{4}$, then it appears that the corollary holds for all even values $m>2$ except for the single value of $m=4$. And in that case, we have $e_V(4) = 1$ but $\pi_U(4) = 4 \neq 2 = \pi_V(4)$; that is, the entry point exists but the periods do not coincide. However, when $p \equiv 2 \pmod{4}$ and $m>2$ is even, the existence of $e_V(m)$ appears to always guarantee that $\pi_U(m) = \pi_V(m)$. Can we prove that the corollary holds for all even values $p$ and $m$ with $m>2$, except in the singular case when $m=4$ and $p \equiv 0 \pmod{4}$?
\end{question}

\begin{question}\label{question:vertical_slice_result_in_q_equal_1_setting}
In Remark~\ref{rem:when_does_vertical_slice_result_hold_for_q_equal_1}, we brought up the observation that ``vertical slice result'' of Theorem~\ref{thm:V_n_order_4_vertical_slice_result} often holds even when the entry point $e_V(m)$ does not exist. For example for the sequence $\left(V_n(7,1)\right)_{n \geq 0}$, when $m=41$ or $m=28$ (see the first two circles in Figure~\ref{fig:vertical_slice_success_and_failure}), $\omega_V(m)$ does not exist, and hence $\omega_V(m) \neq 4$. In the $m=41$ case, the vertical slice result holds; that is, $6+35 = 41$, $7+34 = 41$, and $2+39 = 41$. While in the $m=28$ case, it does not hold.
\begin{figure}[h]
    \centering
    \setcounter{ga}{1}
    \begin{tikzpicture}[scale=0.4,every node/.style={scale=0.5}]
        \def \n {10}
        \def \radius {5}
        \def \innerRadius {4.2}
        \draw [dashed, blue] (-\innerRadius,0) -- (\innerRadius,0);
        \draw [dashed, blue] (0,-\innerRadius) -- (0,\innerRadius);
        \node at (0,0) [draw=red, fill=white] {\LARGE$\bm{m=41}$};
        \draw circle(\radius)
              foreach\s in{1,...,\n}{
                  (-360/\n*\s-90:-\radius)node[fill=red,circle, inner sep=2pt]{}
                  node[anchor=-360/\n*\s-90]{\large\blue{$\widetilde{V}_{\s}\ifnum\s=\n\relax=\widetilde{V}_0\fi \ifnum\s=666\relax=\widetilde{C}_e\fi \ifnum\s=666\relax=\widetilde{C}_{3e}\fi$}}
              };
        \draw foreach\v in {7,6,35,34,39,34,35,6,7,2}{
                  (-360/\n*\thega-89:-\innerRadius)
                  node[anchor=-360/\n*\thega-89]{\large{$\mathbf{\v}$}\stepcounter{ga}}
              };
    \end{tikzpicture}
    \;
    \setcounter{ga}{1}
    \begin{tikzpicture}[scale=0.4,every node/.style={scale=0.5}]
        \def \n {12}
        \def \radius {5}
        \def \innerRadius {4.2}
        \draw [dashed, blue] (-\innerRadius,0) -- (\innerRadius,0);
        \draw [dashed, blue] (0,-\innerRadius) -- (0,\innerRadius);
        \node at (0,0) [draw=red, fill=white] {\LARGE$\bm{m=28}$};
        \draw circle(\radius)
              foreach\s in{1,...,\n}{
                  (-360/\n*\s-90:-\radius)node[fill=red,circle, inner sep=2pt]{}
                  node[anchor=-360/\n*\s-90]{\large\blue{$\widetilde{V}_{\s}\ifnum\s=\n\relax=\widetilde{V}_0\fi \ifnum\s=666\relax=\widetilde{C}_e\fi \ifnum\s=666\relax=\widetilde{C}_{3e}\fi$}}
              };
        \draw foreach\v in {7,19,14,23,7,26,7,23,14,19,7,2}{
                  (-360/\n*\thega-89:-\innerRadius)
                  node[anchor=-360/\n*\thega-89]{\large{$\mathbf{\v}$}\stepcounter{ga}}
              };
    \end{tikzpicture}
    \;
    \setcounter{ga}{1}
    \begin{tikzpicture}[scale=0.4,every node/.style={scale=0.5}]
        \def \n {10}
        \def \radius {5}
        \def \innerRadius {4.2}
        \draw [dashed, blue] (-\innerRadius,0) -- (\innerRadius,0);
        \draw [dashed, blue] (0,-\innerRadius) -- (0,\innerRadius);
        \node at (0,0) [draw=red, fill=white] {\LARGE$\bm{m=33}$};
        \draw circle(\radius)
              foreach\s in{1,...,\n}{
                  (-360/\n*\s-90:-\radius)node[fill=red,circle, inner sep=2pt]{}
                  node[anchor=-360/\n*\s-90]{\large\blue{$\widetilde{V}_{\s}\ifnum\s=\n\relax=\widetilde{V}_0\fi \ifnum\s=666\relax=\widetilde{C}_e\fi \ifnum\s=666\relax=\widetilde{C}_{3e}\fi$}}
              };
        \draw foreach\v in {7,14,25,29,13,29,25,14,7,2}{
                  (-360/\n*\thega-89:-\innerRadius)
                  node[anchor=-360/\n*\thega-89]{\large{$\mathbf{\v}$}\stepcounter{ga}}
              };
    \end{tikzpicture}
    \caption{Fundamental periods of $(V_n(7,1) \pmod{m})_{n=0}^{\pi_V(m)}$ for $m = 41, 28, 33$.}
    \label{fig:vertical_slice_success_and_failure}
\end{figure}
From those two examples, one might conjecture that the vertical slice result holds when $\pi_V(m) \equiv 2 \pmod{4}$ and fails when $\pi_V(m) \equiv 0 \pmod{4}$ since $\pi_V(41)=10$ and $\pi_V(28)=12$. But, that is not the case because when $m=33$ (see the third circle in Figure~\ref{fig:vertical_slice_success_and_failure}), $\pi_V(33) = 10 \equiv 2 \pmod{4}$ holds yet the result fails. Can we find necessary and sufficient conditions for the ``vertical slice result'' to hold for $\LucasFormalVSeq$ modulo $m$ with $q=1$ and $\omega_V \neq 4$?
\end{question}


\section*{Acknowledgments}
The authors thank Christian Ballot for helpful conversations that led to Theorem~\ref{thm:Christian_Ballot_equality_of_periods} and Corollary~\ref{cor:Sufficiency_condition_for_equality_of_periods}. We also thank the many participants of the 21st International Fibonacci Conference who provided feedback to author Mbirika on an earlier version of this paper. Moreover, we are very grateful to the anonymous referee who gave a multitude of comments helping us improve both the exposition and accuracy of the paper. 

\begin{center}
    \includegraphics[width=3.5in]{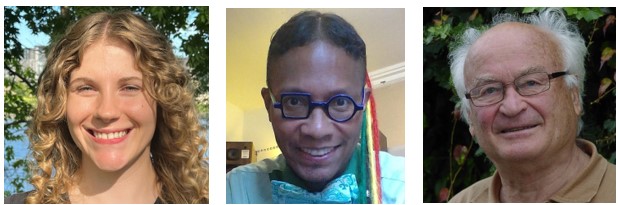}\\
    \textbf{Authors}: Morgan Fiebig\footnote{Morgan Fiebig is a student in the Research Emphasis Mathematics program at the University of Wisconsin-Eau Claire in the USA. She plans to enter a PhD program in Biostatistics in Fall 2025.}, aBa Mbirika\footnote{aBa Mbirika is a Professor of Mathematics at the University of Wisconsin-Eau Claire in the USA.}, and J\"urgen Spilker\footnote{J\"urgen Spilker is an Emeritus Professor of Mathematics at the University of Freiburg in Germany.}
\end{center}


\end{document}